\documentclass[11pt]{amsart}
\usepackage{amsmath, amsthm, amscd, amsfonts,amssymb, graphicx}
\usepackage{accents}
\usepackage{xcolor}
\usepackage{hyperref}
\newlength{\dhatheight}

\newcommand{\bea}{\begin{eqnarray*}}
\newcommand{\eea}{\end{eqnarray*}}

\newcommand{\beq}{\begin{equation}}
\newcommand{\eeq}{\end{equation}}

\newcommand{\dd}{{\rm d}}

\newcommand{\ii}{\mathrm{i}}

\newtheorem{theorem}{Theorem}[section]
\newtheorem{lemma}[theorem]{Lemma}
\theoremstyle{definition}

\newtheorem{proposition}[theorem]{Proposition}
\newtheorem{corollary}[theorem]{Corollary}

\theoremstyle{remark}
\newtheorem{remark}[theorem]{Remark}

\numberwithin{equation}{section}

\begin{document}

\title[Non-Abelian Fourier Series on $\mathbb{Z}^2\backslash SE(2)$]{Non-Abelian Fourier Series on $\mathbb{Z}^2\backslash SE(2)$}

\author[A. Ghaani Farashahi]{Arash Ghaani Farashahi$^{1,*}$}
\address{$^1$Department of Mechanical Engineering, National University of Singapore,
9 Engineering Drive 1, Singapore 117575.}
\email{arash.ghaanifarashahi@nus.edu.sg}
\email{ghaanifarashahi@outlook.com}

\author[G.S. Chirikjian]{Gregory S. Chirikjian$^2$}
\address{$^2$Department of Mechanical Engineering, National University of Singapore,
9 Engineering Drive 1, Singapore 117575.}
\email{mpegre@nus.edu.sg}

%\email{}

\date{\today}

\thanks{$^*$Corresponding author}
\thanks{E-mail addresses: arash.ghaanifarashahi@nus.edu.sg (Arash Ghaani Farashahi) and mpegre@nus.edu.sg (Gregory S. Chirikjian)}

\begin{abstract}
This paper discusses computational structure of coefficients of non-Abelian Fourier series on the right coset space $\mathbb{Z}^2\backslash SE(2)$ expressed in the trigonometric basis,  where $SE(2)$ is the group of handedness preserving Euclidean isometries of the plane and $\mathbb{Z}^2$ denotes the discrete subgroup of translations of the orthogonal (square) lattice in $\mathbb{R}^2$.   Assume that $\mu$ is the finite $SE(2)$-invariant measure on the right coset space $\mathbb{Z}^2\backslash SE(2)$,  normalized with respect to Weil's formula.  We present a constructive computational characterization including discrete sampling of non-Abelian Fourier matrix elements on $SE(2)$ for coefficients of $\mu$-square integrable functions on $\mathbb{Z}^2\backslash SE(2)$ with respect to the concrete trigonometric basis.   
The paper is concluded with discussion of the method for non-Abelian Fourier coefficients of convolutions on $\mathbb{Z}^2\backslash SE(2)$. 
\end{abstract}

\maketitle

%\tableofcontents

\section{{\bf Introduction}}

Special Euclidean groups are classical non-Abelian and non-compact semi-direct product groups used for modeling of the rigid body motions on Euclidean spaces. These semi-direct product Lie groups play significant roles in mathematical analysis,  geometry,   and physics \cite{A.Z.ACHA, Bag.Mer.Pac.Tay, Fuhr,  Kisil.Book, Kisil.JPA}. Over the last decades,  constructive approximation techniques including Fourier analysis on special Euclidean groups and their discretizations have obtained significant attention in both theoretical and applied areas, see \cite{Ba.Gi,  PI5,  Du,  AGHF.GSC.ITSF,  AGHF.GSC.JAT,  Les,  DR.SIAM.JC.2007,DR.ACHA.1995,YY.ITSF.2007,YY.IPI.2007}.  
The right coset space associated to the orthogonal lattice $\mathbb{Z}^2$ in the 2D special Euclidean group $SE(2)$, denoted as $\mathbb{Z}^2\backslash SE(2)$, appears as the configuration space in applications related to computational science and engineering including computer vision, robotics, mathematical crystallography, computational biology, and material science \cite{GSC.JFAA.2000, PI3, ABK.GSC.ACHA.2000, W.Z.M.GSC.IEEE.TC}.
Most such applications involve localized convolutions of functions on $SE(2)$ which can be reformulated as group convolution of $SE(2)$ on $\mathbb{Z}^2\backslash SE(2)$.  

The group of translations of the orthogonal lattice $\mathbb{Z}^2$ is not a normal subgroup of $SE(2)$. 
Consequently,  the right coset space $\mathbb{Z}^2\backslash SE(2)$ has no locally compact group structure.  
Therefore, the classical notion of non-Abelian Fourier transform compatible with group convolution on $\mathbb{Z}^2\backslash SE(2)$ is meaningless. However, there are still both algebraic and geometric structures on $\mathbb{Z}^2\backslash SE(2)$,  namely  the transitive right smooth action of the Lie group $SE(2)$ on the $3$-dimensional manifold $\mathbb{Z}^2\backslash SE(2)$,  making the right coset space $\mathbb{Z}^2\backslash SE(2)$ into a compact homogeneous space which can be viewed geometrically as the $3$-torus. 
Our goal in this paper is to discuss a constructive technique to endow the right coset space $\mathbb{Z}^2\backslash SE(2)$ with a unified structure that allows for constructive approximation of such convolutions while benefiting from both group structure of $SE(2)$ and compactness of the right coset space $\mathbb{Z}^2\backslash SE(2)$ by employing the concrete basis of trigonometric polynomials on the $3$-torus.  

The unified mathematical theory of Fourier analysis on homogeneous spaces of compact groups is addressed in \cite{AGHF.GGD1, AGHF.MMJ} and references therein.  In the case of canonical homogeneous spaces of semi-direct product groups with Abelian normal factor, a concrete approach to the relative Fourier analysis is discussed in \cite{AGHF.JKMS}.  The later approaches strongly benefit from some conditions on the homogeneous space which do not hold for the case of the homogeneous space $\mathbb{Z}^2\backslash SE(2)$ and hence a different Fourier method is required for the right coset space $\mathbb{Z}^2\backslash SE(2)$.  In \cite{AGHF.GSC.PAMQ}, a systematic approach for non-Abelian Fourier series on the right coset spaces of discrete co-compact subgroups of $SE(2)$ introduced.  

Suppose that $\mathbf{\Gamma}$ is a discrete co-compact subgroup of $SE(2)$. Let $\mu$ be the finite $SE(2)$-invariant measure on the right coset space $\mathbf{\Gamma}\backslash SE(2)$ normalized with respect to Weil's formula 
and $\mathcal{E}(\mathbf{\Gamma}):=(\psi_\ell)_{\ell\in\mathbb{I}}$ be a countable orthonormal basis for the Hilbert function space $L^2(\mathbf{\Gamma}\backslash SE(2),\mu)$.  Assume that $K$ is a compact subset of $SE(2)$ and $f\in\mathcal{C}_c(SE(2))$ is supported in $K$ such that $\widehat{f}\in\mathcal{H}^1(0,\infty)$.  It is shown that  
\begin{equation}\label{FourierInv.PAMQ}
\widetilde{f}=\sum_{\ell\in\mathbb{I}}c_\ell\psi_\ell,
\vspace{-0.2cm}
\end{equation}
where the series converges in the mean of the $L^2(\mathbf{\Gamma}\backslash SE(2))$ with 
\begin{equation}\label{Fourier.C.PAMQ}
c_\ell:=\int_0^\infty\mathrm{tr}\left[\widehat{f}(p)Q_{K}^\ell(p)\right]p\dd p,\hspace{.75cm}{\rm where}\ \ Q^\ell_{K}(p):=\int_{K}\overline{\psi_\ell(\Gamma g)} U_p(g)\dd g,
\end{equation}
for every $p>0$, $\gamma\in\mathbf{\Gamma}$ and $\ell\in\mathbb{I}$, see Theorem 3.1 and Proposition 3.2 of \cite{AGHF.GSC.PAMQ}.

We here explore analytical aspects of the above constructive Fourier type approximation method for functions $f\in L^1\cap L^2(SE(2))$ such that $\widetilde{f}\in L^2(\mathbf{\Gamma}\backslash SE(2))$ and $\widehat{f}\in\mathcal{H}^1(0,\infty)$,  when $\mathbf{\Gamma}:=\mathbb{Z}^2$ is the set of translational motions of the orthogonal lattice in $\mathbb{R}^2$ and $\mathcal{E}(\mathbb{Z}^2)$ is the concrete trigonometric basis for the Hilbert function space $L^2(\mathbb{Z}^2\backslash SE(2),\mu)$ given by the trigonometric polynomials.

This article organized in 6 sections.  
Section \ref{sec:2} is devoted to fixing notation and gives preliminaries of non-Abelian Fourier analysis on 
the unimodular group $SE(2)$ and classical properties of functional analysis on 
the right coset space $\mathbb{Z}^2\backslash SE(2)$. In section \ref{sec:3},  we present a generalized version of formula (\ref{Fourier.C.PAMQ}) for a larger class of functions and  investigate structure of the trigonometric basis on the right coset space $\mathbb{Z}^2\backslash SE(2)$.  
Section \ref{sec:4} discusses analytical properties of non-Abelian Fourier integral matrix element on $SE(2)$. 
We then introduce a constructive method for analysing non-Abelian  Fourier coefficients in $L^2(\mathbb{Z}^2\backslash SE(2))$  using the analytical theory of harmonic analysis of Hankel transforms.   As the main result we
present a constructive closed form for computational structure of coefficients of square-integrable functions $\widetilde{f}$ expressed in the trigonometric basis in terms of non-Abelian Fourier matrix elements of $f$ on $SE(2)$.   The paper is concluded with application of the discussed analytic method for non-Abelian Fourier coefficients  of convolutions on $\mathbb{Z}^2\backslash SE(2)$.  

\section{\bf{Preliminaries and Notation}}
\label{sec:2}

Throughout this section we shall present preliminaries and the notation. 

We denote the set of non-negative (resp.  non-positive) real numbers 
$[0,\infty)$ (resp.  $(-\infty,0]$) by $\mathbb{R}_+$ (resp.  $\mathbb{R}_-$) and the set of non-zero real numbers by $\mathbb{R}^*$.  Then $\mathbb{R}^*_+$ (resp.  $\mathbb{R}^*_-$) denotes $(0,\infty)$ (resp.  $(-\infty,0)$). 
For $\mathbf{x},\mathbf{y}\in\mathbb{R}^2$,  $\langle\mathbf{x},\mathbf{y}\rangle$ is the dot product (standard inner product).

\subsection{Bessel functions}
The Bessel function $J_m(x)$ satisfies the following integral representations
\begin{equation}\label{B.i.R}
J_m(x)=\frac{\ii^{-m}}{2\pi}\int_0^{2\pi}e^{\ii x\cos\theta}e^{-\ii m\theta}\dd\theta=\frac{1}{2\pi}\int_0^{2\pi}e^{\ii x\sin\theta}e^{-\ii m\theta}\dd\theta,
\end{equation}
which is equivalent to the following Jacobi-Anger expansions, formula (4.9.6) of \cite{AAR99}, 
\begin{equation}\label{JA}
e^{\ii x\cos\theta}=\sum_{m=-\infty}^{\infty}\ii^mJ_m(x)e^{\ii m\theta},
\hspace{1cm}
e^{\ii x\sin\theta}=\sum_{m=-\infty}^{\infty}J_m(x)e^{\ii m\theta}.
\vspace{-0.2cm}
\end{equation}
Assume that $m\in\mathbb{Z}$ and $N\in\mathbb{N}$.
Let $H_m^N:L^2(\mathbb{R}^*_+,r\dd r)\to L^2(\mathbb{R}^*_+,r\dd r)$ be the bounded linear operator given by 
\begin{equation}
H_m^N(v)(r):=\int_0^Nv(p)J_m(pr)p\dd p,
\end{equation}
for every $v\in L^2(\mathbb{R}^*_+,r\dd r)$ and $r\in\mathbb{R}^*_+$.
The Hankel transform of order $m$ of $v\in L^2(\mathbb{R}^*_+,r\dd r)$, denoted by $H_m^\infty(v)$,  is then defined as the unique limit in the Hilbert space $L^2(\mathbb{R}^*_+,r\dd r)$, satisfying 
\[
\lim_{N\to\infty}\int_0^\infty\left|H_m^\infty(v)(r)-\int_0^Nv(p)J_m(pr)p\dd p\right|^2r\dd r=0.
\]
The Hankel transform $H_m^\infty:L^2(\mathbb{R}^*_+,r\dd r)\to L^2(\mathbb{R}^*_+,r\dd r)$ is a unitary operator and $(H_m^\infty)^{-1}=H_m^\infty$,  see \cite{Kerr}.  If $v\in L^2(\mathbb{R}^*_+,r\dd r)$ and the integral 
\begin{equation*}
\int_0^\infty v(p)J_m(pr)p\dd p,
\end{equation*} 
is converges absolutely for a.e. $r\in\mathbb{R}^*_+$, then  
\begin{equation}\label{Hank.2.abs}
H_m^\infty(v)(r)=\int_0^\infty v(p)J_m(pr)p\dd p,\ \ \ {\rm for\ a.e.}\ \ r\in\mathbb{R}^*_+.
\end{equation}
In particular, if $v\in L^1\cap L^2(\mathbb{R}^*_+,r\dd r)$ then $H_m^\infty(v)$ satisfies the explicit formula (\ref{Hank.2.abs})
for a.e. $r\in\mathbb{R}^*_+$, where the integral (\ref{Hank.2.abs}) converges absolutely for every $r\in\mathbb{R}^*_+$. In addition, if $v\in\mathcal{C}_c^\infty(0,\infty)$ then 
\[
v(r)=\int_0^\infty H_m^\infty(v)(p)J_m(pr)p\dd p,
\]
for every $r\in (0,\infty)$, see Lemma 2.7 of \cite{Bet.Ste}.

\subsection{Harmonic analysis on $SE(2)$}
The 2D special Euclidean motion group,  denoted $SE(2)$,  can be realized as the semi-direct product of the Abelian group $\mathbb{R}^2$ with the 2D special orthogonal group $SO(2)$,  that is $SE(2)=\mathbb{R}^2\rtimes SO(2)$.
The group element $g\in SE(2)$ is then denoted as $g=(\mathbf{x},\mathbf{R})$ where $\mathbf{x}\in\mathbb{R}^2$ and $\mathbf{R}\in SO(2)$.  For elements $g=(\mathbf{x},\mathbf{R})$ and 
$g'=(\mathbf{x}',\mathbf{R}')\in SE(2)$ the group operation is 
\[
g\circ g'=(\mathbf{x}+\mathbf{R}\mathbf{x}',\mathbf{RR}'),
\]
with the inverse 
\[
g^{-1}=(-\mathbf{R}^T\mathbf{x},\mathbf{R}^T)=(-\mathbf{R}^{-1}\mathbf{x},\mathbf{R}^{-1}).
\]
The group $SE(2)$ can be faithfully represented by the group of matrices of the form 
\begin{equation}\label{gxyt}
g(x_1,x_2,\theta):=\hspace{-0.1cm}\left(\hspace{-0.1cm}\begin{array}{ccc}
\cos\theta & -\sin\theta & x_1 \\ 
\sin\theta & \cos\theta & x_2 \\ 
0 & 0 & 1
\end{array}\hspace{-0.1cm}\right)
\hspace{-0.1cm}=\hspace{-0.1cm}\left(\hspace{-0.1cm}\begin{array}{ccc}
\cos\theta & -\sin\theta & a\cos\phi \\ 
\sin\theta & \cos\theta & a\sin\phi \\ 
0 & 0 & 1
\end{array}\hspace{-0.1cm}\right)\hspace{-0.1cm}=:g(a,\phi,\theta).
\end{equation}
The group $SE(2)$ is non-Abelian and unimodular.  The normalized Haar measure on $SE(2)$ is given by 
\[
\dd g=\frac{1}{4\pi^2}\dd\mathbf{x}\dd\theta=\frac{1}{4\pi^2}\dd x_1\dd x_2\dd\theta=\frac{1}{4\pi^2}a\dd a\dd\phi \dd\theta.
\]
Suppose $f_j\in L^1(SE(2))$ with $j\in\{1,2\}$. The convolution $f_1\star f_2\in L^1(SE(2))$ is 
\[
(f_1\star f_2)(g)=\int_{SE(2)}f_1(h)f_2(h^{-1}\circ g)\dd h,\hspace{1cm}{\rm for}\ g\in SE(2).
\]
The Lie group $SE(2)$ is solvable,  and so standard theory for characterizing unitary representations of solvable Lie groups could be employed, see \cite{M.Sugiu}.
The unitary dual $\widehat{SE(2)}$, that is the set of equivalence classes of irreducible unitary representations of $SE(2)$,   can be identified via the following characterization  
\begin{equation*}
\widehat{SE(2)}:=\left\{\chi_n:n\in\mathbb{Z}\right\}\cup\left\{U_p:p> 0\right\}.
\end{equation*}
If $n\in\mathbb{Z}$ then the character $\chi_n:SE(2)\to\mathbb{T}$ is defined by 
\begin{equation}
\chi_n(g):=e^{\ii n\theta},
\end{equation}
for every $g=(\mathbf{t},\mathbf{R}_\theta)\in SE(2)$, where $\mathbb{T}:=\{z\in\mathbb{C}:|z|=1\}$ is the circle group.

If $p>0$ then the irreducible unitary representation $U_p:SE(2)\to\mathcal{U}(L^2(\mathbb{S}^1))$ is defined by $g\mapsto U_p(g)$ where 
\begin{equation}\label{Upg}
[U_p(g)v](\mathbf{u}):=e^{\ii p\langle\mathbf{u},\mathbf{t}\rangle}v(\mathbf{R}_\theta^T\mathbf{u}),
\end{equation}
for every $g=(\mathbf{t},\mathbf{R}_\theta)\in SE(2)$, $v\in L^2(\mathbb{S}^1)$,  and $\mathbf{u}\in\mathbb{S}^1:=\{\mathbf{x}\in\mathbb{R}^2:\|\mathbf{x}\|_2=1\}$. \\

If $f\in L^1(SE(2))$ then the Fourier transform at $n\in\mathbb{Z}$ is given by 
\begin{equation}
\widehat{f}[n]:=\int_{SE(2)}f(g)\overline{\chi_n(g)}\dd g.
\end{equation}
If $f_j\in L^1(SE(2))$ with $j\in\{1,2\}$ then the convolution formula for $n\in\mathbb{Z}$ is 
\begin{equation}\label{star.n}
\widehat{f_1\star f_2}[n]=\widehat{f_1}[n]\widehat{f_2}[n].
\end{equation}

If $f\in L^1(SE(2))$ then the non-Abelian Fourier transform at $p>0$ is defined as an operator on $L^2(\mathbb{S}^1)$ given in the weak sense by the following integral  
\begin{equation}\label{Hfp}
\widehat{f}(p):=\int_{SE(2)}f(g)U_p(g^{-1})\dd g.
\end{equation}
If $f_j\in L^1(SE(2))$ with $j\in\{1,2\}$ then the convolution formula for $p>0$ reads as  
\begin{equation}\label{star.p}
\widehat{f_1\star f_2}(p)=\widehat{f_2}(p)\widehat{f_1}(p).
\end{equation}

Suppose that $(\mathbf{e}_k)_{k\in\mathbb{Z}}$ is the standard orthonormal basis sequence for $L^2(\mathbb{S}^1)$, where $\mathbf{e}_k:\mathbb{S}^1\to\mathbb{C}$
is defined by $\mathbf{e}_k(\mathbf{u}_\alpha):=e^{\ii k\alpha}$, 
 for $k\in\mathbb{Z}$ and $\mathbf{u}_\alpha=(\cos\alpha,\sin\alpha)\in\mathbb{S}^1$ with $\alpha\in[0,2\pi]$.
For every $g\in SE(2)$ and $p>0$, the matrix elements of the linear operator $U_p(g)$ according to the basis $(\mathbf{e}_k)_{k\in\mathbb{Z}}$, are expressed as 
\[
\mathrm{u}_{m,n}(g,p)=\langle U_p(g)\mathbf{e}_m,\mathbf{e}_n\rangle_{L^2(\mathbb{S}^1)},
\]
for every $n,m\in\mathbb{Z}$, where 
\[
\langle u,v\rangle_{L^2(\mathbb{S}^1)}:=\frac{1}{2\pi}\int_0^{2\pi}u(\alpha)\overline{v(\alpha)}\dd\alpha,\hspace{1cm}{\rm for}\ u,v\in L^2(\mathbb{S}^1).
\]
For every $q\ge 1$,  assume that $\mathcal{H}^q(0,\infty)$ is the Banach space consisting of all measurable fields of bounded linear operators $F$ on $(0,\infty)$ with
\[
\|F\|_{\mathcal{H}^q(0,\infty)}^q:=\int_0^\infty\|F(p)\|_q^qp\dd p<\infty,
\]
where for a bounded linear operator $T$, the Schatten $q$-norm of $T$
is $\|T\|_q^q:=\mathrm{tr}[|T|^q]$ and $|T|^2:=T^*T$, see \cite{Lip}.

The Parseval formula on $SE(2)$ associated to non-Abelian Fourier transform (\ref{Hfp})  is  
\[
\int_{SE(2)}|f(g)|^2\dd g=\int_0^\infty\|\widehat{f}(p)\|_{2}^2p\dd p,
\]
and the non-Abelian Fourier inversion formula is 
\begin{equation}\label{iNFT}
f(g)=\int_0^\infty\mathrm{tr}\left[\widehat{f}(p)U_p(g)\right]p\dd p,
\end{equation}
for $f\in L^1\cap L^2(SE(2))$ and $g\in SE(2)$, see \cite{M.Sugiu}. 

\subsection{Harmonic analysis on $\mathbb{Z}^2\backslash SE(2)$}
Let $\mathbb{Z}^2:=\{(\gamma_1,\gamma_2):\gamma_1,\gamma_2\in\mathbb{Z}\}$ be the group of translational isometries of the  orthogonal lattice in $\mathbb{R}^2$. Then $\mathbb{Z}^2$ is a discrete subgroup of $SE(2)$ 
with the counting measure as the Haar measure. Also,  the right coset space $\mathbb{Z}^2\backslash SE(2):=\{\mathbb{Z}^2 g:g\in SE(2)\}$ is compact as a homogeneous space which the Lie group $SE(2)$ acts on it from the right.
The classical aspects of abstract harmonic analysis on locally compact homogeneous spaces are quite well studied by several authors, see \cite{FollH,  AGHF.IJM, HR1, HR.JS} and classical references therein.

The function space $\mathcal{C}(\mathbb{Z}^2\backslash SE(2))$, that is the set of all continuous functions on $\mathbb{Z}^2\backslash SE(2)$, consists of all functions 
$\widetilde{f}$, where 
$f\in\mathcal{C}_c(SE(2))$ and
\begin{equation}\label{5.1}
\widetilde{f}(\mathbb{Z}^2 g):=\sum_{\gamma\in\mathbb{Z}^2}f(\gamma\circ g),
\end{equation}
for all $g\in SE(2)$,  where each $\gamma=(\gamma_1,\gamma_2)\in\mathbb{Z}^2$ considered as the element of $SE(2)$ of the form $g(\gamma_1,\gamma_2,0)$ in (\ref{gxyt}).

Let $\mu$ be a Radon measure on the right coset space $\mathbb{Z}^2\backslash SE(2)$ and $h\in SE(2)$. The right translation $\mu_h$ of $\mu$ is defined by $\mu_h(E):=\mu(E\circ h)$, for all Borel subsets $E$ of $\mathbb{Z}^2\backslash SE(2)$, where $E\circ h:=\{\mathbb{Z}^2 g\circ h:\mathbb{Z}^2 g\in E\}$.  
The measure $\mu$ is called $SE(2)$-invariant if $\mu_h=\mu$, for all $h\in SE(2)$.
Since $SE(2)$ is unimodular, $\mathbb{Z}^2$ is discrete and $\mathbb{Z}^2\backslash SE(2)$ is compact,  there exists a unique finite $SE(2)$-invariant measure $\mu$ on the right coset space $\mathbb{Z}^2\backslash SE(2)$ satisfying the following Weil's formula  
\begin{equation}\label{TH.m}
\int_{\mathbb{Z}^2\backslash SE(2)}\widetilde{f}(\mathbb{Z}^2 g)\dd\mu(\mathbb{Z}^2 g)=\int_{SE(2)}f(g)\dd g,
\end{equation}
for every $f\in L^1(SE(2))$, see Theorem 2.49 of \cite{FollH}. In this case,  $\mu$ is called as the normalized $SE(2)$-invariant measure on the right coset space $\mathbb{Z}^2\backslash SE(2)$.

\section{\bf{Non-Abelian Fourier analysis in $L^2(\mathbb{Z}^2\backslash SE(2))$}}\label{sec:3}

Throughout this section,  we discuss non-Abelian Fourier analysis for square integrable functions on the right coset space of the orthogonal lattice $\mathbb{Z}^2$ in $SE(2)$.  

Assume that $\mu$ is the finite $SE(2)$-invariant measure on the right coset space $\mathbb{Z}^2\backslash SE(2)$ normalized with respect to Weil's formula (\ref{TH.m}).  Then the inner-product in $L^2(\mathbb{Z}^2\backslash SE(2),\mu)$ is given by 
\[
\langle\psi,\varphi\rangle:=\int_{\mathbb{Z}^2\backslash SE(2)}\psi(\mathbb{Z}^2g)\overline{\varphi(\mathbb{Z}^2g)}\dd\mu(\mathbb{Z}^2 g),\hspace{1cm}{\rm for}\ \ \psi,\varphi\in L^2(\mathbb{Z}^2\backslash SE(2),\mu).
\]
Next we present a generalization of the formula (\ref{Fourier.C.PAMQ}) for a larger class of functions.

\begin{theorem}\label{Main.L2.coeff}
Assume $f\in L^1\cap L^2(SE(2))$ such that $\widehat{f}\in\mathcal{H}^1(0,\infty)$ and $\widetilde{f}\in L^2(\mathbb{Z}^2\backslash SE(2))$.  Suppose $\Omega\subset SE(2)$ is a fundamental domain of $\mathbb{Z}^2$ in $SE(2)$
and $\varphi\in\mathcal{C}(\mathbb{Z}^2\backslash SE(2))$.  Then  
\begin{equation}\label{FourierInvQ0.L2.coeff}
\langle\widetilde{f},\varphi\rangle=\sum_{\gamma\in\mathbb{Z}^2}\int_0^\infty\mathrm{tr}\left[\widehat{f}(p)Q_{\gamma\circ\Omega}^{\overline{\varphi}}(p)\right]p\dd p,
\end{equation}
where 
\[
Q^{\overline{\varphi}}_{\gamma\circ\Omega}(p):=\int_{\gamma\circ\Omega}\overline{\varphi(\mathbb{Z}^2 g)}U_p(g)\dd g\hspace{1cm}{\rm for}\ p>0.
\]
\end{theorem}
\begin{proof}
Since $f\in L^1\cap L^2(SE(2))$,  using the formula (\ref{iNFT}), we have 
\begin{equation}\label{alt.in.main.K}
f(g)=\int_0^\infty\mathrm{tr}\left[\widehat{f}(p)U_p(g)\right]p\dd p,
\end{equation}
for $g\in SE(2)$. Therefore, using Weil's formula and (\ref{alt.in.main.K}), we get 
\begin{align*}
\langle\widetilde{f},\varphi\rangle&=\int_{\mathbb{Z}^2\backslash SE(2)}\widetilde{f}(\mathbb{Z}^2 g)
\overline{\varphi(\mathbb{Z}^2 g)}\dd\mu(\mathbb{Z}^2 g)
=\int_{SE(2)}f(g)\overline{\varphi(\mathbb{Z}^2 g)}\dd g
\\&=\sum_{\gamma\in\mathbb{Z}^2}\int_{\Omega}f(\gamma\circ\omega)\overline{\varphi(\mathbb{Z}^2 \omega)}\dd\omega=\sum_{\gamma\in\mathbb{Z}^2}\int_{\gamma\circ\Omega}f(g)\overline{\varphi(\mathbb{Z}^2 g)}\dd g
\\&=\sum_{\gamma\in\mathbb{Z}^2}\int_{\gamma\circ\Omega}\left(\int_0^\infty\mathrm{tr}\left[\widehat{f}(p)U_p(g)\right]p\dd p\right)\overline{\varphi(\mathbb{Z}^2 g)}\dd g.
\end{align*}
Let ${\mathbf{\gamma}}\in\mathbb{Z}^2$ be given.  Invoking the assumption $\widehat{f}\in\mathcal{H}^1(0,\infty)$,  we achieve that 
\begin{align*}
\int_{\gamma\circ\Omega}\int_0^\infty\hspace{-0.1cm}|\mathrm{tr}\left[\widehat{f}(p)U_p(g)\right]||\varphi(\mathbb{Z}^2 g)|p\dd p\dd g
&\le\int_{\gamma\circ\Omega}\int_0^\infty\|\widehat{f}(p)U_p(g)\|_1|\varphi(\mathbb{Z}^2 g)|p\dd p\dd g
\\&\le\int_{\gamma\circ\Omega}\int_0^\infty\|\widehat{f}(p)\|_1\|U_p(g)\||\varphi(\mathbb{Z}^2 g)|p\dd p\dd g 
\\&=\int_{\gamma\circ\Omega}\int_0^\infty\|\widehat{f}(p)\|_1|\varphi(\mathbb{Z}^2 g)|p\dd p\dd g
\\&=\|\widehat{f}\|_{\mathcal{H}^1}\hspace{-0.1cm}\int_{\gamma\circ\Omega}\hspace{-0.1cm}|\varphi(\mathbb{Z}^2 g)|\dd g
\hspace{-0.1cm}=\|\widehat{f}\|_{\mathcal{H}^1}\|\varphi\|_{1}\hspace{-0.1cm}<\infty,
\end{align*}
implying that  
\[
\int_{\gamma\circ\Omega}\hspace{-0.1cm}\left(\int_0^\infty\mathrm{tr}\left[\widehat{f}(p)U_p(g)\right]p\dd p\right)\overline{\varphi(\mathbb{Z}^2 g)}\dd g
=\int_0^\infty\hspace{-0.1cm}\left(\int_{\gamma\circ\Omega}\mathrm{tr}\left[\widehat{f}(p)U_p(g)\right]\overline{\varphi(\mathbb{Z}^2 g)}\dd g\right)p\dd p.
\]
Thus, we achieve 
\begin{equation}\label{mainT.alt0.K}
\langle\widetilde{f},\varphi\rangle=\sum_{\gamma\in\mathbb{Z}^2}\int_0^\infty\left(\int_{\gamma\circ\Omega}\mathrm{tr}\left[\widehat{f}(p)U_p(g)\right]\overline{\varphi(\mathbb{Z}^2 g)}\dd g\right)p\dd p.
\end{equation}
In addition,  $\widehat{f}\in\mathcal{H}^1(0,\infty)$ guarantees that the non-Abelian Fourier integral $\widehat{f}(p)$ is a trace-class operator for a.e. $p>0$. So,  we obtain   
\begin{align*}
\int_{\gamma\circ\Omega}\sum_{m\in\mathbb{Z}}|\varphi(\mathbb{Z}^2 g)||\langle\widehat{f}(p)  U_p(g)\mathbf{e}_m,\mathbf{e}_m\rangle|\dd g
&=\hspace{-0.1cm}\int_{\gamma\circ\Omega}\hspace{-0.1cm}|\varphi(\mathbb{Z}^2 g)|\hspace{-0.05cm}\left(\sum_{m\in\mathbb{Z}}|\langle\widehat{f}(p)  U_p(g)\mathbf{e}_m,\mathbf{e}_m\rangle|\hspace{-0.1cm}\right)\hspace{-0.1cm}\dd g
\\&\le\int_{\gamma\circ\Omega}|\varphi(\mathbb{Z}^2 g)|\|\widehat{f}(p)  U_p(g)\|_1\dd g
\\&\le\|\widehat{f}(p)\|_1\left(\int_{\gamma\circ\Omega}|\varphi(\mathbb{Z}^2 g)|\dd g\right)<\infty,
\end{align*}
for a.e. $p>0$. Thus,  by Fubini's Theorem,  we get   
\begin{align*}
\mathrm{tr}\left[\widehat{f}(p)Q_{\gamma\circ\Omega}^{\overline{\varphi}}(p)\right]
&=\sum_{m\in\mathbb{Z}}\langle\widehat{f}(p)Q_{\gamma\circ\Omega}^{\overline{\varphi}}(p)\mathbf{e}_m,\mathbf{e}_m\rangle
=\sum_{m\in\mathbb{Z}}\langle Q_{\gamma\circ\Omega}^{\overline{\varphi}}(p)\mathbf{e}_m,\widehat{f}(p)^*\mathbf{e}_m\rangle
\\&=\sum_{m\in\mathbb{Z}}\int_{\gamma\circ\Omega}\hspace{-0.1cm}\overline{\varphi(\mathbb{Z}^2 g)}\langle U_p(g)\mathbf{e}_m,\widehat{f}(p)^*\mathbf{e}_m\rangle \dd g=\sum_{m\in\mathbb{Z}}\int_{\gamma\circ\Omega}\hspace{-0.1cm}\overline{\varphi(\mathbb{Z}^2 g)}\langle\widehat{f}(p)  U_p(g)\mathbf{e}_m,\mathbf{e}_m\rangle\dd g
\\&=\int_{\gamma\circ\Omega}\sum_{m\in\mathbb{Z}} \overline{\varphi(\mathbb{Z}^2 g)}\langle\widehat{f}(p)  U_p(g)\mathbf{e}_m,\mathbf{e}_m\rangle\dd g
=\int_{\gamma\circ\Omega}\hspace{-0.1cm}\overline{\varphi(\mathbb{Z}^2 g)}\left(\sum_{m\in\mathbb{Z}}\langle\widehat{f}(p)  U_p(g)\mathbf{e}_m,\mathbf{e}_m\rangle\right)\dd g
\\&=\int_{\gamma\circ\Omega}\overline{\varphi(\mathbb{Z}^2 g)}\mathrm{tr}\left[\widehat{f}(p)U_p(g)\right]\dd g,
\end{align*}
for a.e. $p\in (0,\infty)$. Then, using (\ref{mainT.alt0.K}), we conclude  
\begin{align*}
\langle\widetilde{f},\varphi\rangle
&=\sum_{\gamma\in\mathbb{Z}^2}\hspace{-0.1cm}\int_0^\infty\hspace{-0.15cm}\left(\int_{\gamma\circ\Omega}\hspace{-.2cm}\mathrm{tr}\left[\widehat{f}(p)U_p(g)\right]\overline{\varphi(\mathbb{Z}^2 g)}\dd g\hspace{-.1cm}\right)\hspace{-.1cm}p\dd p
=\sum_{{\gamma}\in\mathbb{Z}^2}\hspace{-.1cm}\int_0^\infty\hspace{-.2cm}\mathrm{tr}\left[\widehat{f}(p)Q_{\gamma\circ\Omega}^{\overline{\varphi}}(p)\right]p\dd p.
\end{align*}
\end{proof}

We then conclude the following version of the formula (\ref{FourierInvQ0.L2.coeff}) in terms of matrix elements. 

\begin{theorem}\label{MainMat}
Suppose $f\in L^1\cap L^2(SE(2))$ such that $\widehat{f}\in\mathcal{H}^1(0,\infty)$ and $\widetilde{f}\in L^2(\mathbb{Z}^2\backslash SE(2))$.  Let $\Omega\subset SE(2)$ be a fundamental domain of $\mathbb{Z}^2$ in $SE(2)$
and $\varphi\in\mathcal{C}(\mathbb{Z}^2\backslash SE(2))$.  Then  
\begin{equation}
\langle\widetilde{f},\varphi\rangle=\sum_{\gamma\in\mathbb{Z}^2}\sum_{m\in\mathbb{Z}}\sum_{n\in\mathbb{Z}}\int_0^\infty\widehat{f}(p)_{n,m}Q_{\gamma\circ\Omega}^{\overline{\varphi}}(p)_{m,n}p\dd p,
\end{equation}
where for $p>0$ we have 
\[
Q^{\overline{\varphi}}_{\gamma\circ\Omega}(p)_{m,n}:=\int_{\gamma\circ\Omega}\overline{\varphi(\mathbb{Z}^2 g)}\mathrm{u}_{m,n}(g;p)\dd g.
\]
\end{theorem}
\begin{proof}
Invoking Theorem \ref{Main.L2.coeff}, we have 
\begin{align*}
\langle\widetilde{f},\varphi\rangle
&=\sum_{{\gamma}\in\mathbb{Z}^2}\int_0^\infty\mathrm{tr}\left[\widehat{f}(p)Q_{\gamma\circ\Omega}^{\overline{\varphi}}(p)\right]p\dd p
\\&=\sum_{\gamma\in\mathbb{Z}^2}\int_0^\infty\sum_{m\in\mathbb{Z}}\sum_{n\in\mathbb{Z}}\widehat{f}(p)_{n,m}Q_{\gamma\circ\Omega}^{\overline{\varphi}}(p)_{m,n}p\dd p
=\sum_{\gamma\in\mathbb{Z}^2}\sum_{m\in\mathbb{Z}}\sum_{n\in\mathbb{Z}}\int_0^\infty\widehat{f}(p)_{n,m}Q_{\gamma\circ\Omega}^{\overline{\varphi}}(p)_{m,n}p\dd p.
\end{align*}
\end{proof}

\subsection{The trigonometric ONB for $L^2(\mathbb{Z}^2\backslash SE(2))$}\label{TrigB}

In this section, we study the concrete orthonormal basis for the Hilbert function space $L^2(\mathbb{Z}^2\backslash SE(2))$ using trigonometric polynomials.

Suppose that $\Omega:=[0,1)^2\times[0,2\pi)$.  Then, $\Omega$ is a fundamental domain for $\mathbb{Z}^2$ in $SE(2)$. We then introduce the orthonormal basis consists of trigonometric polynomials for square integrable functions on $\Omega$. 

For every integral vector $\mathbf{k}=(k_1,k_2,k_3)\in\mathbb{Z}^3$, define the function 
$\phi_\mathbf{k}:\Omega\to\mathbb{C}$ by 
\begin{equation}
\phi_\mathbf{k}(x,y,\theta):=\sqrt{2\pi}\exp\left(2\pi \ii(k_1x+k_2y)\right)\exp(\ii k_3\theta),\hspace{1cm}\ {\rm for}\ (x,y,\theta)\in\Omega.
\end{equation}

\begin{proposition}
Let $\Omega=[0,1)^2\times[0,2\pi)$. Then, $\mathcal{E}:=(\phi_\mathbf{k})_{\mathbf{k}\in\mathbb{Z}^3}$ is an orthonormal basis for the Hilbert function space $L^2(\Omega,\dd g)$.
\end{proposition}
\begin{proof}
Let $\mathbf{n}:=(n_1,n_2,n_3),\mathbf{m}:=(m_1,m_2,m_3)\in\mathbb{Z}^3$ be integral vectors. 
We then have
\begin{align*}
&\langle\phi_{\mathbf{n}},\phi_{\mathbf{m}}\rangle_{L^2(\Omega,\dd g)}
=\frac{1}{4\pi^2}\int_0^1\int_0^1\int_0^{2\pi}\phi_{\mathbf{n}}(x,y,\theta)\overline{\phi_{\mathbf{m}}(x,y,\theta)}\dd x\dd y\dd\theta
\\&=\frac{1}{2\pi}\int_0^1\int_0^1\int_0^{2\pi}\exp\left(2\pi \ii((n_1-m_1)x+(n_2-m_2)y)\right)\exp(\ii (n_3-m_3)\theta)\dd x\dd y\dd\theta
\\&=\left(\frac{1}{2\pi}\int_0^{2\pi}\exp(\ii (n_3-m_3)\theta)\dd\theta\right)\int_0^1\int_0^1\exp\left(2\pi \ii((n_1-m_1)x+(n_2-m_2)y)\right)\dd x\dd y
\\&=\delta_{n_3,m_3}\left(\int_0^1\exp\left(2\pi \ii(n_1-m_1)x\right)\dd x\right)\left(\int_0^1\exp\left(2\pi \ii(n_2-m_2)y\right)\dd y\right)
\\&=\delta_{n_1,m_1}\delta_{n_2,m_2}\delta_{n_3,m_3}=\delta_{\mathbf{n},\mathbf{m}}.
\end{align*} 
In addition,  completeness of $\mathcal{E}$ in $L^2(\Omega,\dd g)$ is straightforward.
\end{proof}

Next we show that integrals on the coset space $\mathbb{Z}^2\backslash SE(2)$ can be computed via integration on $\Omega$.

\begin{lemma}\label{intO}
Let $\mu$ be the normalized $SE(2)$-invariant measure on the right coset space $\mathbb{Z}^2\backslash SE(2)$ and $\Omega:=[0,1)^2\times[0,2\pi)$.  Suppose $\psi\in L^1(\mathbb{Z}^2\backslash SE(2),\mu)$. Then 
\[
\int_{\mathbb{Z}^2\backslash SE(2)}\psi(\mathbb{Z}^2 g)\dd\mu(\mathbb{Z}^2 g)=\int_{\Omega}\psi(\mathbb{Z}^2g)\dd g.
\]
\end{lemma}
\begin{proof}
Suppose that $E_\Omega$ is the characteristic function of $\Omega$. 
Since $\Omega$ is a fundamental domain of $\mathbb{Z}^2$ in $SE(2)$,  we get $\sum_{\gamma\in\mathbb{Z}^2}E_\Omega(\gamma\circ g)=1$,
for every $g\in SE(2)$.  Therefore,  using Weil's formula (\ref{TH.m}), we obtain 
\begin{align*}
\int_{\Omega}\psi(\mathbb{Z}^2g)\dd g
&=\hspace{-0.05cm}\int_{SE(2)}\hspace{-0.05cm}E_{\Omega}(g)\psi(\mathbb{Z}^2g)\dd g.
\\&=\hspace{-0.05cm}\int_{\mathbb{Z}^2\backslash SE(2)}\hspace{-0.15cm}\left(\hspace{-0.05cm}\sum_{\gamma\in\mathbb{Z}^2}\hspace{-0.15cm}E_\Omega(\gamma\circ g)\hspace{-0.05cm}\right)\psi(\mathbb{Z}^2 g)\dd\mu(\mathbb{Z}^2 g)
=\hspace{-0.15cm}\int_{\mathbb{Z}^2\backslash SE(2)}\hspace{-0.15cm}\psi(\mathbb{Z}^2 g)\dd\mu(\mathbb{Z}^2 g).
\end{align*}
\end{proof}

For $\mathbf{k}\in\mathbb{Z}^3$, one can consider the function $\phi_\mathbf{k}$ as a function on the right coset space $\mathbb{Z}^2\backslash SE(2)$. From, now on we shall denote these functions by $\psi_\mathbf{k}$. In details,  
\begin{equation}\label{MainBasisFun}
\psi_\mathbf{k}(\mathbb{Z}^2 g(x,y,\theta)):=\phi_\mathbf{k}(x,y,\theta)=\sqrt{2\pi}\exp\left(2\pi \ii(k_1x+k_2y)\right)\exp(\ii k_3\theta),
\end{equation}
for all $(x,y,\theta)\in\Omega$ and $\mathbf{k}:=(k_1,k_2,k_3)\in\mathbb{Z}^3$.

We then conclude this section by the following results about the structure of the Hilbert space $L^2(\mathbb{Z}^2\backslash SE(2),\mu)$.

\begin{proposition}
{\it Let $\mu$ be the normalized $SE(2)$-invariant measure on the right coset space $\mathbb{Z}^2\backslash SE(2)$ and $\Omega:=[0,1)^2\times[0,2\pi)$.  
\begin{enumerate}
\item The Hilbert function spaces $L^2(\mathbb{Z}^2\backslash SE(2),\mu)$ and $L^2(\Omega,\dd g)$ are isometric isomorphic.
\item $(\psi_\mathbf{k})_{\mathbf{k}\in\mathbb{Z}^3}$ is an orthonormal basis for the Hilbert space $L^2(\mathbb{Z}^2\backslash SE(2),\mu)$. 
\end{enumerate}
}\end{proposition}
\begin{proof}
(1) Suppose  $H:L^2(\mathbb{Z}^2\backslash SE(2),\mu)\to L^2(\Omega,\dd g)$ is given by $\psi\mapsto H\psi$, where 
\[
H\psi(x,y,\theta):=\psi(\mathbb{Z}^2g(x,y,\theta)),
\]
for all $\psi\in L^2(\mathbb{Z}^2\backslash SE(2),\mu)$ and $(x,y,\theta)\in\Omega$. Then using Lemma \ref{intO},  $H$ is a linear unitary operator, i.e. it is an isometric isomorphism of Hilbert spaces.

(2) It is straightforward from (1). 
\end{proof}

\section{\bf{Matrix elements of non-Abelian Fourier integral}}
\label{sec:4}
In this section, we study analytical aspects of non-Abelian Fourier integral matrix element.

In polar coordinates, the matrix elements $\mathrm{u}_{m,n}(g;p)$ and $\widehat{f}(p)_{n,m}$ can be computed via the following closed forms.

\begin{proposition}
{\it Let $m,n\in\mathbb{Z}$ and $p>0$. 
\begin{enumerate}
\item For every $g(r,\phi,\theta)\in SE(2)$, we have 
\begin{equation}\label{upmat.mn}
\mathrm{u}_{m,n}(g(r,\phi,\theta);p)=\ii^{n-m}e^{-\ii(m\theta+(n-m)\phi)}J_{n-m}(pr).
%=\ii^{m-n}e^{-\ii(m\theta+(n-m)\phi)}J_{m-n}(pr).
\end{equation}
\item For every $f\in L^1(SE(2))$, we have 
\begin{equation}\label{fnm}
\widehat{f}(p)_{n,m}=\frac{\ii^{m-n}}{4\pi^2}\int_{0}^{2\pi}\hspace{-0.15cm}\int_0^{2\pi}\hspace{-0.15cm}\int_0^\infty\hspace{-0.15cm} f(r,\phi,\theta)e^{\ii(m\theta+(n-m)\phi)}J_{n-m}(rp)r\dd r\dd\phi\dd\theta.
\end{equation}
\end{enumerate}
}\end{proposition}
\begin{proof}
(1) Applying (\ref{Upg}),  in polar coordinates, we obtain 
\begin{equation}\label{up.em}
[U_p(g(r,\phi,\theta))\mathbf{e}_m](\alpha)=e^{\ii pr\langle\mathbf{u}_\alpha,\mathbf{u}_\phi\rangle}\mathbf{e}_m(\alpha-\theta)=e^{\ii pr\cos(\alpha-\phi)}e^{-\ii m\theta}e^{\ii m\alpha}.
\end{equation}
Using (\ref{B.i.R}) and (\ref{up.em}),  we achieve  
\begin{align*}
\mathrm{u}_{m,n}(g(r,\phi,\theta);p)
&=\frac{1}{2\pi}\hspace{-0.1cm}\int_0^{2\pi}\hspace{-0.15cm}[U_p(g(r,\phi,\theta))\mathbf{e}_m](\alpha)\overline{\mathbf{e}_n(\alpha)}\dd\alpha
=\frac{e^{-\ii m\theta}}{2\pi}\hspace{-0.1cm}\int_0^{2\pi}\hspace{-0.2cm}e^{\ii pr\cos(\alpha-\phi)}e^{\ii(m-n)\alpha}\dd\alpha
\\&=\frac{e^{-\ii m\theta}e^{\ii(m-n)\phi}}{2\pi}\int_0^{2\pi}e^{\ii pr\cos\alpha}e^{-\ii(n-m)\alpha}\dd\alpha
=\ii^{n-m}e^{-\ii(m\theta+(n-m)\phi)}J_{n-m}(pr).
\end{align*}
(2) Invoking (1), we get 
\begin{align*}
\widehat{f}(p)_{n,m}
&=\langle\widehat{f}(p)\mathbf{e}_n,\mathbf{e}_m\rangle_{L^2(\mathbb{S}^1)}=\int_{SE(2)}f(g)\langle U_{p}(g^{-1})\mathbf{e}_n,\mathbf{e}_m\rangle\dd g
\\&=\int_{SE(2)}f(g)\langle\mathbf{e}_n, U_{p}(g)\mathbf{e}_m\rangle \dd g=\int_{SE(2)}f(g)\overline{\mathrm{u}_{m,n}(g;p)}\dd g
\\&=\frac{1}{4\pi^2}\int_{0}^{2\pi}\int_0^{2\pi}\int_0^\infty f(r,\phi,\theta)\overline{\mathrm{u}_{m,n}(g(r,\phi,\theta);p)}r\dd r\dd\phi\dd\theta
\\&=\frac{\ii^{m-n}}{4\pi^2}\int_{0}^{2\pi}\int_0^{2\pi}\int_0^\infty f(r,\phi,\theta)e^{\ii(m\theta+(n-m)\phi)}J_{n-m}(pr)r\dd r\dd\phi\dd\theta.
\end{align*}
\end{proof}

For every $f\in L^1(SE(2))$ and $m,n\in\mathbb{Z}$, let $\widehat{f}_{n,m}:(0,\infty)\to\mathbb{C}$ be the function given by $p\mapsto\widehat{f}(p)_{n,m}$. 
Then applying the closed form (\ref{fnm}),  one can investigate analytical properties of matrix elements $\widehat{f}_{n,m}$.  

\begin{proposition}\label{fnm.cts.bd}
{\it Let $f\in L^1(SE(2))$,  and $m,n\in\mathbb{Z}$. Then 
\begin{enumerate}
\item $\widehat{f}_{n,m}\in\mathcal{C}(\mathbb{R}^*_+)$.
\item $\lim_{p\to 0^+}\widehat{f}(p)_{n,m}=\delta_{n,m}\widehat{f}[-m]=\delta_{n,m}\widehat{f}[-n]$.
\item $\|\widehat{f}_{n,m}\|_\infty\le \|f\|_{L^1(SE(2))}\|J_{n-m}\|_\infty$.
\end{enumerate}
}\end{proposition}
\begin{proof}
(1) Suppose that $p\in(0,\infty)$ and $(p_k)\subset (0,\infty)$ is a sequence converging to $p$.  Using (\ref{fnm}), we have 
\[
\widehat{f}(p_k)_{n,m}=\frac{\ii^{m-n}}{4\pi^2}\int_{0}^{2\pi}\int_0^{2\pi}\int_0^\infty f(r,\phi,\theta)e^{\ii(m\theta+(n-m)\phi)}J_{n-m}(rp_k)r\dd r\dd\phi\dd\theta.
\] 
 For every $k\in\mathbb{N}$, let 
\[
f_k(r,\phi,\theta):=\frac{\ii^{m-n}}{4\pi^2} f(r,\phi,\theta)e^{\ii(m\theta+(n-m)\phi)}J_{n-m}(rp_k).
\]
Then $(f_k)\subset L^1(SE(2))$ and $|f_k|\le|f|$ for every $k\in\mathbb{N}$.  In addition,  
\[
\lim_kf_k(r,\phi,\theta)=\frac{\ii^{m-n}}{4\pi^2} f(r,\phi,\theta)e^{\ii(m\theta+(n-m)\phi)}J_{n-m}(rp).
\]
So, using dominated convergence theorem, we get 
\begin{align*}
\lim_k\widehat{f}(p_k)_{n,m}=\lim_k\int_{SE(2)}f_k(g)\dd g=\widehat{f}(p)_{n,m}.
\end{align*}
(2) Applying dominated convergence theorem in (\ref{fnm}),  we get 
\begin{align*}
\lim_{p\to 0^+}\widehat{f}(p)_{n,m}&=\frac{\ii^{m-n}}{4\pi^2}\lim_{p\to 0^+}\int_{0}^{2\pi}\int_0^{2\pi}\int_0^\infty f(r,\phi,\theta)e^{\ii(m\theta+(n-m)\phi)}J_{n-m}(pr)r\dd r\dd\phi\dd\theta
\\&=\frac{\ii^{m-n}}{4\pi^2}\hspace{-0.12cm}\int_{0}^{2\pi}\hspace{-0.12cm}\int_0^{2\pi}\hspace{-0.12cm}\int_0^\infty\hspace{-0.2cm}f(r,\phi,\theta)e^{\ii(m\theta+(n-m)\phi)}\left(\lim_{p\to 0^+}J_{n-m}(pr)\right)r\dd r\dd\phi\dd\theta
\\&=J_{n-m}(0)\frac{\ii^{m-n}}{4\pi^2}\int_{0}^{2\pi}\int_0^{2\pi}\int_0^\infty f(r,\phi,\theta)e^{\ii(m\theta+(n-m)\phi)}r\dd r\dd\phi\dd\theta
\\&=\delta_{n,m}\frac{\ii^{m-n}}{4\pi^2}\int_{0}^{2\pi}\int_0^{2\pi}\int_0^\infty f(r,\phi,\theta)e^{\ii(m\theta+(n-m)\phi)}r\dd r\dd\phi\dd\theta
\\&=\frac{\delta_{n,m}}{4\pi^2}\int_{0}^{2\pi}\int_0^{2\pi}\int_0^\infty\hspace{-0.1cm} f(r,\phi,\theta)e^{\ii m\theta}r\dd r\dd\phi\dd\theta=\delta_{n,m}\widehat{f}[-m]=\delta_{n,m}\widehat{f}[-n].
\end{align*} 
(3) Suppose that $p>0$.  Then,  using (\ref{fnm}),  we have 
\begin{align*}
|\widehat{f}(p)_{n,m}|&\le \frac{1}{4\pi^2}\int_{0}^{2\pi}\int_0^{2\pi}\int_0^\infty |f(r,\phi,\theta)||J_{n-m}(rp)|r\dd r\dd\phi\dd\theta
\\&\le \frac{\|J_{n-m}\|_\infty}{4\pi^2}\int_{0}^{2\pi}\int_0^{2\pi}\int_0^\infty |f(r,\phi,\theta)|r\dd r\dd\phi\dd\theta\le \|f\|_{L^1(SE(2))}\|J_{n-m}\|_\infty.
\end{align*}
So, we obtain 
\[
\|\widehat{f}_{n,m}\|_\infty=\sup_{p\in(0,\infty)}|\widehat{f}(p)_{n,m}|\le\|f\|_{L^1(SE(2))}\|J_{n-m}\|_\infty.
\]
\end{proof}

\begin{lemma}\label{J2}
Let $p\in(0,\infty)$ be given.  Then 
\[
\sum_{n\in\mathbb{Z}}|J_n(p)|^2=1.
\]
\end{lemma}
\begin{proof}
Let $u_p:[0,2\pi]\to\mathbb{C}$ be given by 
$u_p(\theta):=e^{\ii p\cos\theta}$.  Then $u_p:[0,2\pi]\to\mathbb{C}$ is continuous and $u_p\in L^2[0,2\pi]$. 
Also,   using integral representation of Bessel functions, formula (\ref{B.i.R}), one can see that $|\widehat{u_p}(n)|=|J_n(p)|$ for every $n\in\mathbb{Z}$. Since $|u_p(\theta)|=1$ for every $\theta\in[0,2\pi]$,  Parseval's formula implies 
\[
\sum_{n\in\mathbb{Z}}|J_n(p)|^2=\sum_{n\in\mathbb{Z}}|\widehat{u_p}(n)|^2=\|u_p\|_2^2=\frac{1}{2\pi}\int_0^{2\pi}|u_p(\theta)|^2\dd\theta=1.
\]
\end{proof}

A function $f\in L^1(SE(2))$ is called {\it admissible} if 
\[
\beta(f):=(\beta_m(f))_{m\in\mathbb{Z}}\in\ell^\infty(\mathbb{Z}),
\]
where $\beta_m(f):=\sum_{n\in\mathbb{Z}}\|\widehat{f}_{n,m}\|_\infty$, 
for every $m\in\mathbb{Z}$. In this case,  we define 
\[
\beta_\infty(f):=\|\beta(f)\|_{\ell^\infty}=\sup_{m\in\mathbb{Z}}\beta_m(f). 
\]
The following proposition presents a connection of differentiability on $SE(2)$ and admissibility. 

\begin{proposition}
{\it Let $f\in L^1(SE(2))$ such that $\partial_\phi f\in L^1(SE(2))$.  Then $f$ is admissible and  
\[
\beta_\infty(f)\le\|f\|_1+ \frac{\pi}{\sqrt{3}}\|\partial_\phi f\|_1.
\] 
}\end{proposition}
\begin{proof}
Suppose $m\in\mathbb{Z}$ is given.  Using (\ref{upmat.mn}) and integration by parts with respect to $\phi$, we get  
\begin{align*}
|\widehat{f}(p)_{n,m}|
&=\frac{|\widehat{\partial_\phi f}(p)_{n,m}|}{|n-m|}
\le\frac{1}{4\pi^2|n-m|}\int_{0}^{2\pi}\int_0^{2\pi}\int_0^\infty\hspace{-0.05cm}|\partial_\phi f(r,\phi,\theta)||J_{n-m}(rp)|r\dd r\dd\phi\dd\theta
\\&\le\frac{\|J_{n-m}\|_\infty}{4\pi^2|n-m|}\int_{0}^{2\pi}\int_0^{2\pi}\int_0^\infty |\partial_\phi f(r,\phi,\theta)|r\dd r\dd\phi\dd\theta
=\frac{\|\partial_\phi f\|_1\|J_{n-m}\|_\infty}{|n-m|},
\end{align*}
for every $p>0$ and $n\not=m$. Therefore, we obtain 
\[
\|\widehat{f}_{n,m}\|_\infty\le \frac{\|\partial_\phi f\|_1\|J_{n-m}\|_\infty}{|n-m|},
\]
for every $n\not=m$.  Then using Cauchy-Schwartz inequality, we achieve 
\begin{align*}
\sum_{n\not=m}\|\widehat{f}_{n,m}\|_\infty
&\le\|\partial_\phi f\|_1\sum_{n\not=m}\frac{\|J_{n-m}\|_\infty}{|n-m|}\\&\le\|\partial_\phi f\|_1\left(\sum_{n\not=m}\|J_{n-m}\|_\infty^2\right)^{1/2}\left(\sum_{n\not=m}\frac{1}{|n-m|^2}\right)^{1/2}\le \frac{\pi}{\sqrt{3}}\|\partial_\phi f\|_1,
\end{align*}
implying that 
\[
\beta_m(f)=\sum_{n\in\mathbb{Z}}\|\widehat{f}_{n,m}\|_\infty=\|\widehat{f}_{m,m}\|_\infty+\sum_{n\not=m}\|\widehat{f}_{n,m}\|_\infty\le\|f\|_1+ \frac{\pi}{\sqrt{3}}\|\partial_\phi f\|_1.
\]
\end{proof}

\begin{corollary}
{\it Every rapidly decreasing function on $SE(2)$ is admissible.  In particular,  every compactly supported smooth function on $SE(2)$ is admissible.}
\end{corollary}

We finish this section by discussing integrability of non-Abelian matrix elements. 

\begin{proposition}\label{fnmL2}
{\it Suppose $f\in L^1\cap L^2(SE(2))$ and $n,m\in\mathbb{Z}$. Then $\widehat{f}_{n,m}\in L^2(\mathbb{R}^*_+,r\dd r)$. 
}\end{proposition}
\begin{proof}
Since $f\in L^1\cap L^2(SE(2))$, we then have $\widehat{f}\in\mathcal{H}^2(0,\infty)$. Thus, we get 
\begin{align*}
|\widehat{f}(p)_{n,m}|^2\le \sum_{k\in\mathbb{Z}}\sum_{l\in\mathbb{Z}}|\widehat{f}(p)_{k,l}|^2=\|\widehat{f}(p)\|_2^2.
\end{align*}
So,  we obtain 
\[
\int_0^\infty|\widehat{f}(p)_{n,m}|^2p\dd p\le\int_0^\infty\|\widehat{f}(p)\|_2^2p\dd p=\|\widehat{f}\|^2_{\mathcal{H}^2(0,\infty)}<\infty.
\] 
\end{proof}
\begin{corollary}\label{sumfnm}
{\it Let $f\in L^1\cap L^2(SE(2))$ and $m\in\mathbb{Z}$. Then
\[
\sum_{n\in\mathbb{Z}}\|\widehat{f}_{n,m}\|_2^2\le\|f\|_2^2.
\]
}\end{corollary}
\begin{proof}
Since $f\in L^1\cap L^2(SE(2))$,  we have 
\begin{align*}
\sum_{n\in\mathbb{Z}}\|\widehat{f}_{n,m}\|^2_2
&=\sum_{n\in\mathbb{Z}}\int_0^\infty|\widehat{f}(p)_{n,m}|^2p\dd p=\int_0^\infty\left(\sum_{n\in\mathbb{Z}}|\widehat{f}(p)_{n,m}|^2\right)p\dd p
\\&\le\int_0^\infty\left(\sum_{k\in\mathbb{Z}}\sum_{n\in\mathbb{Z}}|\widehat{f}(p)_{n,k}|_2^2\right)p\dd p=\int_0^\infty\|\widehat{f}(p)\|_2^2p\dd p=\|f\|_{2}^2.
\end{align*}
\end{proof}
\begin{corollary}\label{fnmL1}
{\it Let $f\in L^1\cap L^2(SE(2))$ with $\widehat{f}\in\mathcal{H}^1(0,\infty)$ and $n,m\in\mathbb{Z}$. Then $\widehat{f}_{n,m}\in L^1\cap L^2(\mathbb{R}^*_+,r\dd r)$. 
}\end{corollary}
\begin{proof}
Since $f\in L^1\cap L^2(SE(2))$,  by Proposition \ref{fnmL2},  we get $\widehat{f}_{n,m}\in L^2(\mathbb{R}^*_+,r\dd r)$.  In addition, we have 
\begin{align*}
|\widehat{f}(p)_{n,m}|\le\left(\sum_{k\in\mathbb{Z}}\sum_{l\in\mathbb{Z}}|\widehat{f}(p)_{k,l}|^2\right)^{1/2}=\|\widehat{f}(p)\|_2\le\|\widehat{f}(p)\|_1,
\end{align*}
implying that 
\[
\int_0^\infty|\widehat{f}(p)_{n,m}|p\dd p\le\int_0^\infty\|\widehat{f}(p)\|_1p\dd p=\|\widehat{f}\|_{\mathcal{H}^1(0,\infty)}<\infty.
\] 
\end{proof}

\section{\bf{Non-Abelian Fourier Series on $\mathbb{Z}^2\backslash SE(2)$}}
\label{sec:5}

Through this section,  we study analytical aspects of non-Abelian Fourier series associated to the concrete basis of trigonometric polynomials introduced in Section \ref{TrigB}.  As the main result of the paper,  we present a constructive closed form for coefficients of the form $\langle\widetilde{f},\psi_{\mathbf{k}}\rangle$ in terms of non-Abelian Fourier matrix elements of $f$. 
 
Suppose that $\Omega:=[0,1)^2\times[0,2\pi)$ is the canonical fundamental domain for $\mathbb{Z}^2$ in $SE(2)$.  We first start by realization of Theorem \ref{MainMat}  in the case of trigonometric basis.  To this end,  we need to make some generalizations in the notions. 

For $(x,y)\in\mathbb{R}^2$,  assume that  $\rho(x,y)\ge 0$ and $0\le\Phi(x,y)<2\pi$ are given by
\begin{equation*}\label{polxy}
x=\rho(x,y)\cos\Phi(x,y),\hspace{2cm}y=\rho(x,y)\sin\Phi(x,y).
\end{equation*}
For $\ell:=(s,t,k)\in\mathbb{R}^2\times\mathbb{Z}$,  let $\varphi_{\ell}:\mathbb{Z}^2\backslash SE(2)\to\mathbb{C}$ be the function given by 
\begin{equation}\label{VPl}
\varphi_{\ell}(\mathbb{Z}^2 g(x,y,\theta)):=e^{2\pi\ii (sx+ty)}e^{\ii k\theta},\hspace{1cm}{\rm for}\ \ \mathbb{Z}^2 g(x,y,\theta)\in\mathbb{Z}^2\backslash SE(2).
\end{equation} 
Let $f\in L^1\cap L^2(SE(2))$ such that $\widehat{f}\in\mathcal{H}^1(0,\infty)$ and $\widetilde{f}\in L^2(\mathbb{Z}^2\backslash SE(2),\mu)$.  For every $k\in\mathbb{Z}$,  suppose that $Y_f^k:\mathbb{R}^2\to\mathbb{C}$ is the function defined by  
\[
Y_f^k(s,t):=\langle\widetilde{f},\varphi_{\ell}\rangle_{L^2(\mathbb{Z}^2\backslash SE(2))}=\int_{SE(2)}f(g)\overline{\varphi_\ell(\mathbb{Z}^2 g)}\dd g.
%=\int_{SE(2)}f(x,y,\theta)e^{-2\pi\ii (sx+ty)}e^{-\ii k\theta}\dd x\dd y\dd\theta.
\]
Invoking Theorem \ref{MainMat},  we have 
\begin{equation}\label{Yfk.rec}
Y_f^k(s,t)=\sum_{\gamma\in\mathbb{Z}^2}\sum_{n=-\infty}^\infty\sum_{m=-\infty}^\infty\int_0^\infty \widehat{f}(p)_{n,m}Q_{\gamma\circ\Omega}^{\ell}(p)_{m,n}p\dd p,
\end{equation}
where 
\[
Q^\ell_{\gamma\circ\Omega}(p)_{m,n}:=\int_{\gamma\circ\Omega}\overline{\varphi_\ell(\mathbb{Z}^2 g)}\mathrm{u}_{m,n}(g;p)\dd g,
\hspace{1cm}{\rm for}\ p>0, \ {\rm and}\ \gamma\in\mathbb{Z}^2.
\]
We here compute a closed form for the matrix elements $Q^\ell_{\gamma\circ\Omega}(p)_{m,n}$.
\begin{proposition}\label{Qkpmn}
{\it Assume $\ell=(s,t,k)\in\mathbb{R}^2\times\mathbb{Z}$. Let $p>0$, $n,m\in\mathbb{Z}$, and $\gamma=(\gamma_1,\gamma_2)\in\mathbb{Z}^2$.
Then
\begin{align*}
Q_{\gamma\circ\Omega}^{\ell}(p)_{m,n}=\frac{\ii^{n+k}\delta_{k,-m}}{2\pi}\int_{\gamma_1}^{\gamma_1+1}\int_{\gamma_2}^{\gamma_2+1}e^{-\ii(n+k)\Phi(x,y)}J_{n+k}(p\rho(x,y))e^{-2\pi \ii(sx+ty)}\dd x\dd y.
\end{align*}
}\end{proposition}
\begin{proof}
We have $\gamma\circ\Omega=[\gamma_1,\gamma_1+1)\times[\gamma_2,\gamma_2+1)\times[0,2\pi)$.  So using (\ref{upmat.mn}), we get 
\begin{align*}
Q_{\gamma\circ\Omega}^{\ell}(p)_{m,n}
&=\int_{\gamma\circ\Omega}\mathrm{u}_{m,n}(g;p)\overline{\psi_\ell(\mathbb{Z}^2 g)}\dd g
\\&=\frac{1}{4\pi^2}\int_{\gamma_1}^{\gamma_1+1}\int_{\gamma_2}^{\gamma_2+1}\int_0^{2\pi}\mathrm{u}_{m,n}(x,y,\theta;p)e^{-2\pi \ii(sx+ty)}e^{-\ii k\theta}\dd x\dd y\dd\theta
\\&=\frac{\ii^{n-m}}{4\pi^2}\hspace{-0.15cm}\int_{\gamma_1}^{\gamma_1+1}\hspace{-0.15cm}\int_{\gamma_2}^{\gamma_2+1}\hspace{-0.15cm}\int_0^{2\pi}e^{-\ii(m\theta+(n-m)\Phi(x,y))}J_{n-m}(p\rho(x,y))e^{-2\pi \ii(s x+ty)}e^{-\ii k\theta}\dd x\dd y\dd\theta
\\&=\frac{\ii^{n-m}\delta_{k,-m}}{2\pi}\int_{\gamma_1}^{\gamma_1+1}\int_{\gamma_2}^{\gamma_2+1}e^{-\ii(n-m)\Phi(x,y)}J_{n-m}(p\rho(x,y))e^{-2\pi \ii(sx+ty)}\dd x\dd y
\\&=\frac{\ii^{n+k}\delta_{k,-m}}{2\pi}\int_{\gamma_1}^{\gamma_1+1}\int_{\gamma_2}^{\gamma_2+1}e^{-\ii(n+k)\Phi(x,y)}J_{n+k}(p\rho(x,y))e^{-2\pi \ii(sx+ty)}\dd x\dd y.
\end{align*}
\end{proof}

\begin{proposition}\label{D}
{\it Let $f\in L^1\cap L^2(SE(2))$ with $\widetilde{f}\in L^2(\mathbb{Z}^2\backslash SE(2))$ and $\widehat{f}\in\mathcal{H}^1(0,\infty)$.  Assume $\ell=(s,t,k)\in\mathbb{R}^2\times\mathbb{Z}$,  $n,m\in\mathbb{Z}$ and $\gamma=(\gamma_1,\gamma_2)\in\mathbb{Z}^2$. Then 
\[
\int_0^\infty \widehat{f}(p)_{n,m}Q_{\gamma\circ\Omega}^{\ell}(p)_{m,n}p\dd p=\frac{\ii^{n+k}\delta_{k,-m}}{2\pi}\hspace{-0.1cm}\int_{\gamma_1}^{\gamma_1+1}\int_{\gamma_2}^{\gamma_2+1}\hspace{-0.1cm}D_{n,k}(f)(x,y)e^{-2\pi \ii(sx+ty)}\dd x\dd y,
\]
where 
\[
D_{n,k}(f)(x,y):=e^{-\ii(n+k)\Phi(x,y)}H_{n+k}^\infty(\widehat{f}_{n,-k})(\rho(x,y)).
\]
}\end{proposition}
\begin{proof}
Invoking Corollary \ref{fnmL1}, we have $\widehat{f}_{n,m}\in L^1\cap L^2(\mathbb{R}^*_+,r\dd r)$.  So, we get 
\begin{align*}
&\int_{\gamma_1}^{\gamma_1+1}\int_{\gamma_2}^{\gamma_2+1}\int_0^\infty|\widehat{f}(p)_{n,m}||J_{n+k}(p\rho(x,y))|p\dd p\dd x\dd y\\&\le\int_{\gamma_1}^{\gamma_1+1}\int_{\gamma_2}^{\gamma_2+1}\int_0^\infty|\widehat{f}(p)_{n,m}|p\dd p\dd x\dd y\le\|\widehat{f}_{n,m}\|_{L^1(\mathbb{R}^*_+,r\dd r)}.
\end{align*}
Then,  using Fubini's Theorem,  we obtain  
\begin{align*}
&\int_0^\infty \widehat{f}(p)_{n,m}Q_{\gamma\circ\Omega}^{\ell}(p)_{m,n}p\dd p
\\&=\frac{\ii^{n+k}\delta_{k,-m}}{2\pi}\hspace{-0.15cm}\int_0^\infty\hspace{-0.15cm}\widehat{f}(p)_{n,m}\hspace{-0.1cm}\left(\int_{\gamma_1}^{\gamma_1+1}\hspace{-0.25cm}\int_{\gamma_2}^{\gamma_2+1}\hspace{-0.32cm}e^{-\ii(n+k)\Phi(x,y)}J_{n+k}(p\rho(x,y))e^{-2\pi \ii(sx+ty)}\dd x\dd y\hspace{-0.1cm}\right)\hspace{-0.1cm}p\dd p
%\\&=\ii^{n+k}\delta_{k,-m}\hspace{-0.15cm}\int_{\gamma_1}^{\gamma_1+1}%\hspace{-0.2cm}\int_{\gamma_2}^{\gamma_2+1}e^{-\ii(n+k)\Phi(x,y)}%\left(\int_0^\infty \widehat{f}(p)_{n,m}J_{n+k}(p\rho(x,y))p\dd p\right)e^{-2\pi %\ii(sx+ty)}\dd x\dd y
\\&=\frac{\ii^{n+k}\delta_{k,-m}}{2\pi}\int_{\gamma_1}^{\gamma_1+1}\int_{\gamma_2}^{\gamma_2+1}e^{-\ii(n+k)\Phi(x,y)}H_{n+k}^\infty(\widehat{f}_{n,m})(\rho(x,y))e^{-2\pi \ii(sx+ty)}\dd x\dd y
\\&=\frac{\ii^{n+k}\delta_{k,-m}}{2\pi}\int_{\gamma_1}^{\gamma_1+1}\int_{\gamma_2}^{\gamma_2+1}e^{-\ii(n+k)\Phi(x,y)}H_{n+k}^\infty(\widehat{f}_{n,-k})(\rho(x,y))e^{-2\pi \ii(sx+ty)}\dd x\dd y.
\end{align*}
\end{proof}
\begin{proposition}\label{S}
{\it Let $f\in L^1\cap L^2(SE(2))$ with $\widetilde{f}\in L^2(\mathbb{Z}^2\backslash SE(2))$ and $\widehat{f}\in\mathcal{H}^1(0,\infty)$.  Assume $\ell=(s,t,k)\in\mathbb{R}^2\times\mathbb{Z}$ and $\gamma=(\gamma_1,\gamma_2)\in\mathbb{Z}^2$.  Then 
\[
\sum_{n=-\infty}^\infty\sum_{m=-\infty}^\infty\int_0^\infty \widehat{f}(p)_{n,m}Q_{\gamma\circ\Omega}^{\ell}(p)_{m,n}p\dd p=\int_{\gamma_1}^{\gamma_1+1}\int_{\gamma_2}^{\gamma_2+1}\hspace{-0.15cm}S_k(f)(x,y)e^{-2\pi \ii(sx+ty)}\dd x\dd y,
\]
with 
\[
S_k(f)(x,y):=\frac{1}{2\pi}\sum_{n=-\infty}^\infty\ii^{n}e^{-\ii n\Phi(x,y)}H_{n}^\infty(\widehat{f}_{n-k,-k})(\rho(x,y)),
\]
where the series converges absolutely for every $(x,y)\in\mathbb{R}^2$. 
}\end{proposition}
\begin{proof}
By Proposition \ref{D},  we get
\begin{align*}
&\sum_{n=-\infty}^\infty\sum_{m=-\infty}^\infty\int_0^\infty \widehat{f}(p)_{n,m}Q_{\gamma\circ\Omega}^{\ell}(p)_{m,n}p\dd p
\\&=\frac{1}{2\pi}\sum_{n=-\infty}^\infty\ii^{n+k}\int_{\gamma_1}^{\gamma_1+1}\int_{\gamma_2}^{\gamma_2+1}e^{-\ii(n+k)\Phi(x,y)}H_{n+k}^\infty(\widehat{f}_{n,-k})(\rho(x,y))e^{-2\pi \ii(sx+ty)}\dd x\dd y
\\&=\frac{1}{2\pi}\sum_{n=-\infty}^\infty\ii^{n}\int_{\gamma_1}^{\gamma_1+1}\int_{\gamma_2}^{\gamma_2+1}e^{-\ii n\Phi(x,y)}H_{n}^\infty(\widehat{f}_{n-k,-k})(\rho(x,y))e^{-2\pi \ii(sx+ty)}\dd x\dd y.
\end{align*} 
Applying H\"older's inequality and Lemma \ref{J2}, we obtain 
\begin{align*}
\sum_{n=-\infty}^\infty\hspace{-0.1cm}|\widehat{f}(p)_{n-k,-k}||J_{n}(p\rho(x,y))|&\le\left(\sum_{n=-\infty}^\infty\hspace{-0.1cm}|\widehat{f}(p)_{n-k,-k}|^2\right)^{1/2}\hspace{-0.2cm}\left(\sum_{n=-\infty}^\infty\hspace{-0.1cm}|J_{n}(p\rho(x,y))|^2\right)^{1/2}
\\&=\left(\sum_{n=-\infty}^\infty|\widehat{f}(p)_{n,-k}|^2\right)^{1/2}\le\|\widehat{f}(p)\|_2<\infty,
\end{align*}
implying that 
\begin{equation}\label{hat.J.2}
\sum_{n=-\infty}^\infty|\widehat{f}(p)_{n-k,-k}||J_{n}(p\rho(x,y))|\le \|\widehat{f}(p)\|_2,
\end{equation}
for a.e. $p\in(0,\infty)$ and every $(x,y)\in [\gamma_1,\gamma_1+1]\times[\gamma_2,\gamma_2+1]$.  Hence,  we achieve 
\[
\int_0^\infty\sum_{n=-\infty}^\infty|\widehat{f}(p)_{n-k,-k}||J_{n}(p\rho(x,y))|p\dd p\le\int_0^\infty\|\widehat{f}(p)\|_2p\dd p\le\int_0^\infty\|\widehat{f}(p)\|_1p\dd p<\infty,
\] 
implying that 
\begin{equation}\label{S.Hn}
\sum_{n=-\infty}^\infty|H_{n}^\infty(\widehat{f}_{n-k,-k})(\rho(x,y))|\le\|\widehat{f}\|_{\mathcal{H}^1(0,\infty)}.
\end{equation}
for every $(x,y)\in [\gamma_1,\gamma_1+1]\times[\gamma_2,\gamma_2+1]$.  
Therefore,  Fubini's Theorem implies that  
\begin{align*}
&\sum_{n=-\infty}^\infty\sum_{m=-\infty}^\infty\int_0^\infty \widehat{f}(p)_{n,m}Q_{\gamma\Omega}^{\ell}(p)_{m,n}p\dd p
\\&=\frac{1}{2\pi}\sum_{n=-\infty}^\infty\ii^{n}\int_{\gamma_1}^{\gamma_1+1}\int_{\gamma_2}^{\gamma_2+1}e^{-\ii n\Phi(x,y)}H_{n}^\infty(\widehat{f}_{n-k,-k})(\rho(x,y))e^{-2\pi \ii(sx+ty)}\dd x\dd y
\\&=\frac{1}{2\pi}\int_{\gamma_1}^{\gamma_1+1}\int_{\gamma_2}^{\gamma_2+1}\left(\sum_{n=-\infty}^\infty\ii^{n}e^{-\ii n\Phi(x,y)}H_{n}^\infty(\widehat{f}_{n-k,-k})(\rho(x,y))\right)e^{-2\pi \ii(sx+ty)}\dd x\dd y.
\end{align*} 
\end{proof}

Next,  one can characterize $Y_f^k(r,\theta)$ as follow.

\begin{theorem}\label{Ykf.alt}
Suppose $f\in L^1\cap L^2(SE(2))$ such that $\widetilde{f}\in L^2(\mathbb{Z}^2\backslash SE(2))$ and $\widehat{f}\in\mathcal{H}^1(0,\infty)$. 
Assume $(r,\theta)\in\mathbb{R}^2$ and $k\in\mathbb{Z}$.  Then 
\[
Y_f^k(r,\theta)=\int_{0}^\infty\left(\sum_{n=-\infty}^\infty e^{-\ii n\theta}J_n(2\pi rp)H_{n}^\infty(\widehat{f}_{n-k,-k})(p)\right)p\dd p.
\]
\end{theorem}
\begin{proof}
Let $\ell:=(s,t,k)=:(r\cos\theta,r\sin\theta,k)$.  Then,  (\ref{Yfk.rec}) and Proposition \ref{S}, implies   
\begin{align*}
Y_f^k(r,\theta)
&=\sum_{\gamma_1\in\mathbb{Z}}\sum_{\gamma_2\in\mathbb{Z}}
\left(\sum_{n=-\infty}^\infty\sum_{m=-\infty}^\infty\int_0^\infty \widehat{f}(p)_{n,m}Q_{\gamma\circ\Omega}^{\ell}(p)_{m,n}p\dd p\right)
\\&=\sum_{\gamma_1\in\mathbb{Z}}\sum_{\gamma_2\in\mathbb{Z}}\int_{\gamma_1}^{\gamma_1+1}\hspace{-0.22cm}\int_{\gamma_2}^{\gamma_2+1}S_k(f)(x,y)e^{-2\pi \ii(sx+ty)}\dd x\dd y
\\&=\frac{1}{2\pi}\int_{-\infty}^\infty\int_{-\infty}^\infty\left(\sum_{n=-\infty}^\infty\ii^{n}e^{-\ii n\Phi(x,y)}H_{n}^\infty(\widehat{f}_{n-k,-k})(\rho(x,y))\right)e^{-2\pi \ii(sx+ty)}\dd x\dd y
\\&=\frac{1}{2\pi}\int_{0}^\infty\left(\int_{0}^{2\pi}\left(\sum_{n=-\infty}^\infty\ii^{n}e^{-\ii n\phi}H_{n}^\infty(\widehat{f}_{n-k,-k})(p)\right)e^{-2\pi\ii p(s\cos\phi+t\sin\phi)}\dd\phi\right)p\dd p.
\end{align*}
Using the integral representation of Bessel functions, formula (\ref{B.i.R}), we have  
\begin{align*}
\frac{1}{2\pi}\int_0^{2\pi}e^{-\ii n\phi}e^{-2\pi \ii p(s\cos\phi+t\sin\phi)}\dd\phi
&=\frac{1}{2\pi}\int_0^{2\pi}e^{-\ii n\phi}e^{-2\pi \ii pr \cos(\phi-\theta)}\dd\phi
\\&=\frac{e^{-\ii n\theta}}{2\pi}\int_0^{2\pi}e^{-\ii n\phi}e^{-2\pi \ii pr \cos\phi}\dd\phi
\\&=\ii ^n e^{-\ii n\theta}J_{-n}(2\pi rp)
= \ii ^{-n}e^{-\ii n\theta}J_{n}(2\pi rp).
\end{align*}
Therefore,  using inequality (\ref{S.Hn}) and Fubini's Theorem,  we get 
\begin{align*}
\int_{0}^{2\pi}\left(\sum_{n=-\infty}^\infty\ii^{n}e^{-\ii n\phi}H_{n}^\infty(\widehat{f}_{n-k,-k})(p)\right)e^{-2\pi r\ii(s\cos\phi+t\sin\phi)}\dd\phi
=\sum_{n=-\infty}^\infty e^{-\ii n\theta}J_n(2\pi rp)H_{n}^\infty(\widehat{f}_{n-k,-k})(p),
\end{align*}
for a.e.  $p\in(0,\infty)$.  Hence, we achieve 
\begin{align*}
Y_f^k(r,\theta)
&=\frac{1}{2\pi}\hspace{-0.1cm}\int_{0}^\infty\hspace{-0.1cm}\left(\int_{0}^{2\pi}\hspace{-0.1cm}\left(\sum_{n=-\infty}^\infty\ii^{n}e^{-\ii n\phi}H_{n}^\infty(\widehat{f}_{n-k,-k})(p)\right)e^{-2\pi\ii p(s\cos\phi+t\sin\phi)}\dd\phi\right)p\dd p
\\&=\int_{0}^\infty\left(\sum_{n=-\infty}^\infty e^{-\ii n\theta}J_n(2\pi rp)H_{n}^\infty(\widehat{f}_{n-k,-k})(p)\right)p\dd p.
\end{align*}
\end{proof}
We here continue this section by introducing a unified closed form for coefficients of functions of the form $\widetilde{f}\in L^2(\mathbb{Z}^2\backslash SE(2),\mu)$ in terms of non-Abelian Fourier matrix elements of $f$ on $SE(2)$,  according to trigonometric basis.
To this end,  we need to prove the following technical lemmas for further investigation of $Y_f^k(r,\theta)$.  

For functions $u:[0,2\pi)\to\mathbb{C}$ and $v:[0,\infty)\to\mathbb{C}$, let $u\otimes v:\mathbb{R}^2\to\mathbb{C}$ be given by  
\[
u\otimes v(r,\theta):=u(\theta)v(r).
\]
\begin{lemma}\label{VnL2}
Let $(v_n)_{n\in\mathbb{Z}}\subset L^2(\mathbb{R}_+,r\dd r)$ such that $\sum_{n\in\mathbb{Z}}\|v_n\|_{2}^2<\infty$.  Assume that $V_n:=\mathbf{e}_{-n}\otimes v_n$ for every $n\in\mathbb{Z}$. Then $V=\sum_{n\in\mathbb{Z}}V_n$ unconditionally converges in $L^2(\mathbb{R}^2,(2\pi)^{-1}\dd\mathbf{x})$ and 
\[
\left\|\sum_{n\in\mathbb{Z}}V_n\right\|_{2}^2=\sum_{n\in\mathbb{Z}}\|v_n\|^2_{2}.
\]
In particular, if $\sum_{n}|v_n(r)|<\infty$ for a.e.  $r\in [0,\infty)$ then 
\[
V(r,\theta)=\sum_{n\in\mathbb{Z}}V_n(r,\theta)\hspace{1cm}\ {\rm for\ \ a.e. }\  (r,\theta)\in\mathbb{R}^2. 
\]
\end{lemma}
\begin{proof}
Let $n,m\in\mathbb{Z}$ be given.  Then 
\begin{align*}
\langle V_n,V_m\rangle
&=\frac{1}{2\pi}\int_0^{2\pi}\int_0^\infty V_n(r,\theta)\overline{V_m(r,\theta)}r\dd r\dd\theta
=\delta_{n,m}\langle v_n,v_m\rangle
=\delta_{n,m}\|v_n\|_{2}^2.
\end{align*}
Therefore, for every finite subset $\mathbb{J}\subset\mathbb{Z}$, we obtain 
\begin{equation}\label{mainJ}
\vspace{-0.1cm}
\left\|\sum_{n\in\mathbb{J}}V_n\right\|_{2}^2=\sum_{n\in\mathbb{J}}\|v_n\|^2_{2}.
\end{equation}
Let $\imath:\mathbb{N}\to\mathbb{Z}$ be an injection and 
$\sigma:\mathbb{Z}\to\mathbb{Z}$ be a bijection.  
For $N\in\mathbb{N}$, let $W_N:=\sum_{k=1}^NV_{\sigma(\imath_k)}$.
Then, using (\ref{mainJ}), we get 
\begin{align*}
\|W_N-W_M\|_2^2=\left\|\sum_{N+1\le k\le M}V_{\sigma(\imath_k)}\right\|_{2}^2=\sum_{N+1\le k\le M}\|v_{\sigma(\imath_k)}\|^2_{2},
\end{align*}
if $M,N\in\mathbb{N}$ and $M>N$.  Therefore,  $(W_N)$ is Cauchy in $L^2(\mathbb{R}^2)$.  Since $L^2(\mathbb{R}^2)$ is a Banach space,  we conclude that $(W_N)$ is convergent in $L^2(\mathbb{R}^2)$ and hence the series $\sum_{k=1}^\infty V_{\sigma(\imath_k)}$ converges in $L^2(\mathbb{R}^2)$.  Since $\sigma$ and $\imath$ were arbitrary, we conclude that the series $\sum_{n\in\mathbb{Z}}V_n$ unconditionally converges in $L^2(\mathbb{R}^2)$.

Suppose $\sum_{n\in\mathbb{Z}}|v_n(r)|<\infty$ for a.e.  $r\in [0,\infty)$.  Let $W(r,\theta):=\sum_{n\in\mathbb{Z}}V_n(r,\theta)$.  Assume that $\imath:\mathbb{N}\to\mathbb{Z}$ is an injection and $\sigma:\mathbb{Z}\to\mathbb{Z}$ is a bijection.  Since $\sum_{n\in\mathbb{Z}}V_n(r,\theta)$ is absolutely convergent for a.e.  $(r,\theta)$,  we achieve that $W(r,\theta)<\infty$ and 
\[
W(r,\theta)=\sum_{k=1}^\infty V_{\sigma(\imath_k)}(r,\theta),\hspace{1cm}\ {\rm for\ \ a.e. }\  (r,\theta)\in\mathbb{R}^2. 
\vspace{-0.1cm}
\]
Let $V:=\sum_{n\in\mathbb{Z}}V_n$ in $L^2(\mathbb{R}^2)$. 
Then there exists a strictly increasing function $\tau:\mathbb{N}\to\mathbb{N}$ such that 
\[
V(r,\theta)=\lim_{N\to\infty}\sum_{k=1}^{\tau(N)}V_{\sigma(\imath_k)}(r,\theta),\hspace{1cm}\ {\rm for\ \ a.e. }\  (r,\theta)\in\mathbb{R}^2. 
\vspace{-0.1cm}
\]
So,  we get $W(r,\theta)=V(r,\theta)$ for a.e.  $(r,\theta)\in\mathbb{R}^2$.
\end{proof}

\begin{proposition}\label{VnL2U}
{\it Let $(v_n)_{n\in\mathbb{Z}}\subset L^2(\mathbb{R}_+,r\dd r)$ be a sequence of bounded continuous functions such that $\sum_{n\in\mathbb{Z}}\|v_n\|_{2}^2<\infty$ and $\sum_{n\in\mathbb{Z}}\|v_n\|_\infty<\infty$. Suppose $V_n:=\mathbf{e}_{-n}\otimes v_n$ for every $n\in\mathbb{Z}$. Then $V(r,\theta):=\sum_{n\in\mathbb{Z}}V_n(r,\theta)$ is a continuous function in $L^2(\mathbb{R}^2)$. 
}\end{proposition}
\begin{proof}
Since $\sum_n\|v_n\|_\infty<\infty$,  we get $\sum_n|v_n(r)|<\infty$ and hence $\sum_{n}e^{-\ii n\theta}v_n(r)$ converges absolutely and uniformly on $\mathbb{R}^2$. In particular,  $\sum_{n}e^{-\ii n\theta}v_n(r)<\infty$,  for $(r,\theta)\in\mathbb{R}^2$.  Let $V(r,\theta):=\sum_{n}e^{-\ii n\theta}v_n(r)$. 
Using Lemma \ref{VnL2}, we get $V\in L^2(\mathbb{R}^2)$.  
Then the Weierstrass M-test, implies that the series $\sum_ne^{-\ii n\theta}v_n(r)$ converges absolutely and uniformly on $\mathbb{R}^2$.    Since each $V_n$ is continuous,  we get that $V$ is continuous.
\end{proof}

\begin{lemma}
{\it Let $(\nu_n)_{n\in\mathbb{Z}}\subset L^2(\mathbb{R}_+,r\dd r)$ such that $\sum_{n\in\mathbb{Z}}\|\nu_n\|_2^2<\infty$.  Then 
\[
\lim_{N\to\infty}\sum_{n\in\mathbb{Z}}\mathbf{e}_{-n}\otimes H_n^N(\nu_n)=\sum_{n\in\mathbb{Z}}\mathbf{e}_{-n}\otimes H_n^\infty(\nu_n),
\]
where the limit and the series converges in $L^2(\mathbb{R}^2)$.
}\end{lemma}
\begin{proof}
For every $N\in\mathbb{N}$ and $n\in\mathbb{Z}$,  suppose that 
$\nu_n^N:=\chi_{(0,N]}\nu_n$,  $v_n^N:=\nu_n-\nu_n^N$,  $h_n^N:=H_n^\infty(v_n^N)$, and $V_n^N:=\mathbf{e}_{-n}\otimes h_n^N$.  Then $\nu_n^N,v_n^N,h_n^N\in L^2(\mathbb{R}_+,r\dd r)$ with $\|h_n^N\|_2=\|v_n^N\|_2\le 2\|\nu_n\|_2$. In addition,  we have 
\[
\sum_{n\in\mathbb{Z}}\left\|h_n^N\right\|_2^2=\sum_{n\in\mathbb{Z}}\left\|v_n^N\right\|_2^2\le 2\sum_{n\in\mathbb{Z}}\|\nu_n\|_2^2<\infty.
\vspace{-0.1cm}
\] 
Let $F_N:=\sum_n\mathbf{e}_{-n}\otimes H_n^N(\nu_n)$ and $F:=\sum_n\mathbf{e}_{-n}\otimes H_n^\infty(\nu_n)$.
Then $F_N,F\in L^2(\mathbb{R}^2)$.
Applying Lemma \ref{VnL2}, for $(v_n^N)$, we get 
\begin{align*}
\|F-F_N\|^2_{2}
&=\left\|\sum_n\mathbf{e}_{-n}\otimes H_n^\infty(\nu_n)-\sum_n\mathbf{e}_{-n}\otimes H_n^N(\nu_n)\right\|^2_2
\\&=\left\|\sum_n\mathbf{e}_{-n}\otimes H_n^\infty(\nu_n)-\sum_n\mathbf{e}_{-n}\otimes H_n^\infty(\nu_n^N)\right\|^2_2
\hspace{-0.15cm}=\left\|\sum_n V_n^N\right\|^2_2\hspace{-0.15cm}=\sum_n\|v_n^N\|_2^2.
\end{align*}
So,  Lebesgue Dominated Theorem implies $\lim_N\|v_n^N\|_2^2=0$ for $n\in\mathbb{Z}$.  Then  
\[
\lim_N\|F_N-F\|^2_{2}=\lim_N\sum_n\|v_n^N\|_2^2=\sum_n\lim_N\|v_n^N\|_2^2=0.
\]
\end{proof}

\begin{corollary}\label{vnMain}
{\it Let $(v_n)_{n\in\mathbb{Z}}\subset L^2(\mathbb{R}_+,r\dd r)$ such that 
$\sum_{n\in\mathbb{Z}}\|v_n\|_2^2<\infty$.  Then 
\[
\lim_{N\to\infty}\sum_n\mathbf{e}_{-n}\otimes H_n^N(H_n^\infty(v_n))=\sum_n\mathbf{e}_{-n}\otimes v_n,
\]
where the limit,  convergence of the series and equality are in the sense of  $L^2(\mathbb{R}^2)$.
}\end{corollary}

\begin{lemma}\label{FN}
{\it Let $(\nu_n)_{n\in\mathbb{Z}}\subset L^2(\mathbb{R}_+,r\dd r)$ such that $\sum_{n\in\mathbb{Z}}\|\nu_n\|_2^2<\infty$ and $N\in\mathbb{N}$.  Suppose $F_N=\sum_{n\in\mathbb{Z}}\mathbf{e}_{-n}\otimes H_n^N(\nu_n)$.  Then  
$F_N(r,\theta)=\sum_{n\in\mathbb{Z}}e^{-\ii n\theta}H_n^N(\nu_n)(r)$ a.e.  on $\mathbb{R}^2$. }
\end{lemma}
\begin{proof}
Suppose $N\in\mathbb{N}$,  $F_N:=\sum_n\mathbf{e}_{-n}\otimes H_n^N(\nu_n)$, and $r\in\mathbb{R}_+$.  Then 
\begin{align*}
|H_n^N(\nu_n)(r)|&\le\int_0^N|\nu_n(p)||J_n(pr)|p\dd p.
\end{align*}
Using Lemma \ref{J2}, we get 
\begin{align*}
\sum_{n\in\mathbb{Z}}|H_n^N(\nu_n)(r)|
&\le\sum_{n\in\mathbb{Z}}\int_0^N|\nu_n(p)||J_n(pr)|p\dd p=\int_0^N\left(\sum_{n\in\mathbb{Z}}|\nu_n(p)||J_n(pr)|\right)p\dd p
\\&\le\int_0^N\left(\sum_{n\in\mathbb{Z}}|\nu_n(p)|^2\right)^{1/2}\left(\sum_{n\in\mathbb{Z}}|J_n(pr)|^2\right)^{1/2}p\dd p
=\int_0^N\left(\sum_{n\in\mathbb{Z}}|\nu_n(p)|^2\right)^{1/2}p\dd p
\\&\le\frac{N}{\sqrt{2}}\left(\int_0^N\sum_{n\in\mathbb{Z}}|\nu_n(p)|^2p\dd p\right)^{1/2}=\frac{N}{\sqrt{2}}\left(\sum_{n\in\mathbb{Z}}\int_0^N|\nu_n(p)|^2p\dd p\right)^{1/2}\le\frac{N}{\sqrt{2}}\left(\sum_{n\in\mathbb{Z}}\|\nu_n\|_2^2\right)^{1/2},
\end{align*}
implying that 
$F_N(r,\theta)=\sum_{n\in\mathbb{Z}}e^{-\ii n\theta}H_n^N(\nu_n)(r)$ a.e.  on $\mathbb{R}^2$. 
\end{proof}

\begin{lemma}\label{HnM}
Let $(u_n)_{n\in\mathbb{Z}}\subset L^2(\mathbb{R}_+,r\dd r)$ such that $\sum_{n\in\mathbb{Z}}\|u_n\|_2^2<\infty$.  Suppose $(r,\theta)\in\mathbb{R}^2$ and $M>0$.  Then
\[
\int_{0}^{M}\left(\sum_{n=-\infty}^\infty e^{-\ii n\theta}J_n(2\pi rp)u_n(p)\right)p\dd p=\sum_{n=-\infty}^\infty e^{-\ii n\theta}H_n^{M}(u_n)(2\pi r).
\]
\end{lemma}
\begin{proof}
For every $p>0$, we have 
\begin{align*}
\sum_{n=-\infty}^\infty\hspace{-0.1cm}|J_n(2\pi rp)||u_n(p)|
&\le\hspace{-0.1cm}\left(\sum_{n=-\infty}^\infty\hspace{-0.15cm}|J_n(2\pi rp)|^2\hspace{-0.05cm}\right)^{1/2}\hspace{-0.2cm}\left(\sum_{n=-\infty}^\infty\hspace{-0.15cm}|u_n(p)|^2\hspace{-0.05cm}\right)^{1/2}
\hspace{-0.4cm}=\left(\sum_{n=-\infty}^\infty\hspace{-0.15cm}|u_n(p)|^2\hspace{-0.05cm}\right)^{1/2}\hspace{-0.4cm}.
\end{align*}
Therefore,  we obtain  
\begin{align*}
\int_{0}^{M}\hspace{-0.15cm}\left(\sum_{n=-\infty}^\infty\hspace{-0.1cm}|J_n(2\pi rp)||u_n(p)|\hspace{-0.05cm}\right)\hspace{-0.05cm}p\dd p&\le \int_0^M\left(\sum_{n=-\infty}^\infty|u_n(p)|^2\right)^{1/2}p\dd p
\\&\le \frac{M}{\sqrt{2}}\left(\int_0^M\sum_{n=-\infty}^\infty|u_n(p)|^2p\dd p\right)^{1/2}\\&\le\frac{M}{\sqrt{2}}\hspace{-0.05cm}\left(\hspace{-0.05cm}\int_0^\infty\hspace{-0.3cm}\sum_{n=-\infty}^\infty\hspace{-0.1cm}|u_n(p)|^2p\dd p\hspace{-0.1cm}\right)^{1/2}
\hspace{-0.45cm}=\frac{M}{\sqrt{2}}\hspace{-0.1cm}\left(\sum_{n=-\infty}^\infty\hspace{-0.2cm}\|u_n\|_2^2\hspace{-0.05cm}\right)^{1/2}\hspace{-0.3cm}.
\end{align*}
Then Fubini's Theorem implies that 
\begin{align*}
\int_{0}^{M}\left(\sum_{n=-\infty}^\infty e^{-\ii n\theta}J_n(2\pi rp)u_n(p)\right)p\dd p
&=\sum_{n=-\infty}^\infty e^{-\ii n\theta}\left(\int_{0}^{M}J_n(2\pi rp)u_n(p)p\dd p\right)
\\&=\sum_{n=-\infty}^\infty e^{-\ii n\theta}H_n^{M}(u_n)(2\pi r).
\end{align*}
\end{proof}

The following theorem presents an explicit point-wise formulation for the function $Y_f^k(r,\theta)$ in terms of non-Abelian Fourier matrix elements of $f$. 

\begin{theorem}\label{mainEXP}
Let $f\in L^1\cap L^2(SE(2))$ be an admissible function such that $\widetilde{f}\in L^2(\mathbb{Z}^2\backslash SE(2))$ and $\widehat{f}\in\mathcal{H}^1(0,\infty)$.  Suppose $(r,\theta)\in\mathbb{R}^2$  and $k\in\mathbb{Z}$ are given.   
\begin{enumerate}
\item If $r>0$ then 
\begin{equation}\label{Yfk}
Y_f^k(r,\theta)=e^{-\ii k\theta}\sum_{n\in\mathbb{Z}}e^{-\ii n\theta}\widehat{f}(2\pi r)_{n,-k},
\end{equation}
where the series converges absolutely and uniformly on $\mathbb{R}^2-\{\mathbf{0}\}$.
\item $Y_f^k(\mathbf{0})=\widehat{f}[k]$.
\end{enumerate}
\end{theorem}
\begin{proof}
For $n\in\mathbb{Z}$,  suppose $v_n(0):=\delta_{n,0}\widehat{f}[k]$ and  $v_n(p):=\widehat{f}(p)_{n-k,-k}$ if $p>0$.  Then using Propositions \ref{fnm.cts.bd} and \ref{fnmL2} we conclude that $v_n$ is a continuous function in $L^2(\mathbb{R}_+,r\dd r)$.  In addition,  according to Corollary \ref{sumfnm}, we get  
\[
\sum_{n\in\mathbb{Z}}\|v_n\|_2^2=\sum_{n\in\mathbb{Z}}\|\widehat{f}_{n-k,-k}\|_2^2=\sum_{n\in\mathbb{Z}}\|\widehat{f}_{n,-k}\|_2^2<\infty. 
\vspace{-0.1cm}
\]
Let $V(r,\theta):=\sum_{n=-\infty}^\infty e^{-\ii n\theta}v_n(r)$.  Then admissibility of $f$ implies that 
\[
\sum_{n\in\mathbb{Z}}\|v_n\|_\infty\le\sum_{n\in\mathbb{Z}}\|\widehat{f}_{n-k,-k}\|_\infty=\sum_{n\in\mathbb{Z}}\|\widehat{f}_{n,-k}\|_\infty=\beta_{-k}(f).
\]
So Proposition \ref{VnL2U} guarantees that $V$ is a continuous function in $L^2(\mathbb{R}^2)$.  
Assume that $u_n:=H_n^\infty(v_n)$, for every $n\in\mathbb{Z}$.  Then 
$\sum_{n\in\mathbb{Z}}\|u_n\|_2^2=\sum_{n\in\mathbb{Z}}\|v_n\|_2^2<\infty$. 
Let $F_N=\sum_{n\in\mathbb{Z}}\mathbf{e}_{-n}\otimes H_n^N(u_n)$.  Then  
Lemma \ref{FN} implies that $F_N(r,\theta)=\sum_{n\in\mathbb{Z}}e^{-\ii n\theta}H_n^N(u_n)(r)$ a.e.  on $\mathbb{R}^2$.  Then, using Corollary \ref{vnMain}, we have 
\[
\lim_NF_N=\sum_{n=-\infty}^\infty\mathbf{e}_{-n}\otimes H_n^\infty(u_n)=V,
\]
where the limit,  convergence of the series and equalities are in the sense of  $L^2(\mathbb{R}^2)$.  Since $\lim_N F_N=V$ in $L^2(\mathbb{R}^2)$, there exists a strictly increasing function $\tau:\mathbb{N}\to\mathbb{N}$ such that $V(r,\theta)=\lim_NF_{\tau(N)}(r,\theta)$ for a.e.  $(r,\theta)\in\mathbb{R}^2$. 
Now suppose that $N>0$.  Applying Lemma \ref{HnM}, for $M:=\tau(N)$, we get  
\[
\int_{0}^{\tau(N)}\left(\sum_{n=-\infty}^\infty e^{-\ii n\theta}J_n(2\pi rp)u_n(p)\right)p\dd p=\sum_{n=-\infty}^\infty e^{-\ii n\theta}H_n^{\tau(N)}(u_n)(2\pi r),
\vspace{-0.1cm}
\]
for every $(r,\theta)\in\mathbb{R}^2$. 
So,  using Theorem \ref{Ykf.alt},  we obtain  
\begin{align*}
Y_f^k(r,\theta)&=\int_{0}^\infty\left(\sum_{n=-\infty}^\infty e^{-\ii n\theta}J_n(2\pi rp)H_{n}^\infty(\widehat{f}_{n-k_3,-k_3})(p)\right)p\dd p
\\&=\lim_N\int_{0}^N\left(\sum_{n=-\infty}^\infty e^{-\ii n\theta}J_n(2\pi rp)u_n(p)\right)p\dd p
\\&=\lim_N\int_{0}^{\tau(N)}\left(\sum_{n=-\infty}^\infty e^{-\ii n\theta}J_n(2\pi rp)u_n(p)\right)p\dd p
\\&=\lim_N\sum_{n=-\infty}^\infty e^{-\ii n\theta}H_n^{\tau(N)}(u_n)(2\pi r )=\lim_N F_{\tau(N)}(2\pi r,\theta)=V(2\pi r,\theta),
\end{align*}
for a.e.  $(r,\theta)\in\mathbb{R}^2$.  
Since $V$ and $Y_f^k$ are continuous,  we get $Y_f^k(r,\theta)=V(2\pi r,\theta)$,  for every $(r,\theta)\in\mathbb{R}^2$.  Therefore,  if $r>0$ then  
\begin{align*}
Y_f^k(r,\theta)
&=\sum_{n\in\mathbb{Z}}e^{-\ii n\theta}v_n(r)=\sum_{n\in\mathbb{Z}}e^{-\ii n\theta}\widehat{f}(2\pi r)_{n-k,-k}
\\&=\sum_{n\in\mathbb{Z}}e^{-\ii(n+k)\theta}\widehat{f}(2\pi r)_{n,-k}=e^{-\ii k\theta}\sum_{n\in\mathbb{Z}}e^{-\ii n\theta}\widehat{f}(2\pi r)_{n,-k}.
\end{align*}
If $r=0$ then 
\[
Y_f^k(\mathbf{0})=\sum_{n\in\mathbb{Z}}e^{-\ii n\theta}v_n(0)=\sum_{n\in\mathbb{Z}}e^{-\ii n\theta}\delta_{n,0}\widehat{f}[k]=\widehat{f}[k].
\] 
\end{proof}

\begin{remark}
It is worth mentioning that statement of Theorem \ref{mainEXP}(2) can be proven directly without employing the analytic technique introduced in Theorem \ref{mainEXP}.  If $r=0$ and $\ell:=(0,0,k)$ then using Weil's formula (\ref{TH.m}) and (\ref{VPl}),  we have  
\begin{align*}
Y_f^k(\mathbf{0})&=\langle\widetilde{f},\varphi_{\ell}\rangle
=\int_{\mathbb{Z}^2\backslash SE(2)}\widetilde{f}(\mathbb{Z}^2 g)\overline{\varphi_{\ell}(\mathbb{Z}^2 g)}\dd\mu(\mathbb{Z}^2 g)
\\&=\int_{SE(2)}f(g)\overline{\varphi_{\ell}(\mathbb{Z}^2 g)}\dd g
=\frac{1}{4\pi^2}\int_{SE(2)}f(g(x,y,\theta))e^{-\ii k_3\theta}\dd x\dd y\dd\theta=\widehat{f}[k_3].
\end{align*} 
\end{remark}

This section is concluded by the following constructive closed form formulation of non-Abelian Fourier coefficients on $\mathbb{Z}^2\backslash SE(2)$ according to trigonometric basis (\ref{TrigB}). 

\begin{theorem}\label{mainCoef}
Let $f\in L^1\cap L^2(SE(2))$ be an admissible function such that $\widetilde{f}\in L^2(\mathbb{Z}^2\backslash SE(2))$ and $\widehat{f}\in\mathcal{H}^1(0,\infty)$.  Suppose $\mathbf{k}:=(k_1,k_2,k_3)\in\mathbb{Z}^3$.  Then
\begin{enumerate}
\item If $\rho(k_1,k_2)\not=0$ then 
\begin{equation}\label{kkk}
\langle\widetilde{f},\psi_\mathbf{k}\rangle=\sqrt{2\pi}e^{-\ii k_3\Phi(k_1,k_2)}\sum_{n\in\mathbb{Z}}e^{-\ii n\Phi(k_1,k_2)}\widehat{f}(2\pi\rho(k_1,k_2))_{n,-k_3},
\end{equation}
where the series converges absolutely.
\item If $\rho(k_1,k_2)=0$ then  
\begin{equation}\label{00k}
\langle\widetilde{f},\psi_\mathbf{k}\rangle=\sqrt{2\pi}\widehat{f}[k_3].
\end{equation}
\end{enumerate}
\end{theorem}
\begin{proof}
(1) Using Theorem \ref{mainEXP}, we have 
\begin{align*}
\langle\widetilde{f},\psi_\mathbf{k}\rangle&=\sqrt{2\pi}Y_f^{k_3}(k_1,k_2)=\sqrt{2\pi}e^{-\ii k_3\Phi(k_1,k_2)}\sum_{n\in\mathbb{Z}}e^{-\ii n\Phi(k_1,k_2)}\widehat{f}(2\pi\rho(k_1,k_2))_{n,-k_3}.
\end{align*}
(2) Let $\mathbf{k}:=(0,0,k_3)\in\mathbb{Z}^3$.  We then obtain  
\[
\langle\widetilde{f},\psi_{\mathbf{k}}\rangle=\sqrt{2\pi}Y_f^{k_3}(\mathbf{0})=\sqrt{2\pi}\widehat{f}[k_3].
\]
\end{proof}

\section{\bf{Non-Abelian Fourier Series of Convolutions on $\mathbb{Z}^2\backslash SE(2)$}}\label{sec:6}

Throughout this section we still assume that $\mu$ is the finite $SE(2)$-invariant measure on the right coset space $\mathbb{Z}^2\backslash SE(2)$ which is normalized with respect to Weil's formula.  We then discuss analytical aspects of non-Abelian Fourier series of convolution of functions on $SE(2)$ by functions on $\mathbb{Z}^2\backslash SE(2)$.

Let $f\in L^1(SE(2))$ and $\psi\in L^1(\mathbb{Z}^2\backslash SE(2),\mu)$.  The convolution of $f$ with $\psi$ is defined as the function $\psi\oslash f:\mathbb{Z}^2\backslash SE(2)\to\mathbb{C}$ given by  
\begin{equation}\label{oslash}
(\psi\oslash f)(\mathbb{Z}^2 g):=\int_{SE(2)} \psi(\mathbb{Z}^2 h)f(h^{-1}\circ g)\dd h,
\end{equation}
for $g\in SE(2)$.
Since $SE(2)$ is a unimodular group, we get 
\[
(\psi\oslash f)(\mathbb{Z}^2 g)=\int_{SE(2)} \psi(\mathbb{Z}^2 g\circ h)f(h^{-1})\dd h,
\]
for $f\in L^1(SE(2))$,  $\psi\in L^1(\mathbb{Z}^2\backslash SE(2),\mu)$,  and $g\in SE(2)$.
It is shown that if $f_j\in L^1(SE(2))$ with $j\in\{1,2\}$ then  
$\widetilde{(f_1\star f_2)}(\mathbb{Z}^2 g)=(\widetilde{f_1}\oslash f_2)(\mathbb{Z}^2 g)$,  for $g\in SE(2)$, see Theorem 4.2 of \cite{AGHF.GSC.PAMQ}.

\begin{lemma}\label{conv12}
Suppose $f_j\in L^1(SE(2))$ with $j\in\{1,2\}$.  If $f_2\in L^2(SE(2))$ then $f_1\star f_2\in L^1\cap L^2(SE(2))$.  
\end{lemma}
\begin{proof}
Since $f_j\in L^1(SE(2))$ with $j\in\{1,2\}$, we get $f_1\star f_2\in L^1(SE(2))$.  If $f_2\in L^2(SE(2))$ then $f_1\star f_2\in L^2(SE(2))$ as well.  Therefore, we conclude that  $f_1\star f_2\in L^1\cap L^2(SE(2))$.  
\end{proof}

\begin{proposition}
{\it Let $f\in L^1(SE(2))$ and $\psi\in L^2(\mathbb{Z}^2\backslash SE(2))$.  Then  $\psi\oslash f\in L^2(\mathbb{Z}^2\backslash SE(2))$ with 
\[
\|\psi\oslash f\|_{L^2(\mathbb{Z}^2\backslash SE(2),\mu)}\le \|f\|_{L^1(SE(2))}\|\psi\|_{L^2(\mathbb{Z}^2\backslash SE(2),\mu)}.
\]
}\end{proposition}
\begin{proof}
Let $f\in L^1(SE(2))$ and $\psi\in L^2(\mathbb{Z}^2\backslash SE(2),\mu)$. We then have
\begin{align*}
\|\psi\oslash f\|_{L^2(\mathbb{Z}^2\backslash SE(2),\mu)}
&=\left(\int_{\mathbb{Z}^2\backslash SE(2)}\left|\int_{SE(2)} \psi(\mathbb{Z}^2 g\circ h)f(h^{-1})\dd h\right|^2\dd\mu(\mathbb{Z}^2 g)\right)^{1/2}
\\&\le\int_{SE(2)}\left(\int_{\mathbb{Z}^2\backslash SE(2)}|\psi(\mathbb{Z}^2 g\circ h)|^2|f(h^{-1})|^2\dd\mu(\mathbb{Z}^2 g)\right)^{1/2}\dd h
\\&=\int_{SE(2)}|f(h^{-1})|\left(\int_{\mathbb{Z}^2\backslash SE(2)}|\psi(\mathbb{Z}^2 g\circ h)|^2\dd\mu(\mathbb{Z}^2 g)\right)^{1/2}\dd h
\\&=\int_{SE(2)}|f(h^{-1})|\left(\int_{\mathbb{Z}^2\backslash SE(2)}|\psi(\mathbb{Z}^2 g)|^2\dd\mu(\mathbb{Z}^2 g)\right)^{1/2}\dd h
\\&=\|\psi\|_{L^2(\mathbb{Z}^2\backslash SE(2),\mu)}\int_{SE(2)}|f(h^{-1})|\dd h
=\|f\|_{1}\|\psi\|_{2}.
\end{align*}
\end{proof}

\begin{corollary}\label{conv2}
{\it Suppose $f_j\in L^1\cap L^2(SE(2))$ with $j\in\{1,2\}$ such that $\widetilde{f_1}\in L^2(\mathbb{Z}^2\backslash SE(2))$. Then  $\widetilde{f_1}\oslash f_2\in L^2(\mathbb{Z}^2\backslash SE(2))$.
}\end{corollary}

The paper is concluded by demonstration of a constructive closed form for formulation of non-Abelian Fourier coefficients of convolution on $\mathbb{Z}^2\backslash SE(2)$ according to the concrete trigonometric basis (\ref{MainBasisFun}). 

\begin{lemma}\label{convMat}
Let $f_j\in L^1(SE(2))$ with $j\in\{1,2\}$. Suppose $p>0$ and $n,k\in\mathbb{Z}$. Then 
\[
\widehat{f_1\star f_2}(p)_{n,k}=\sum_{m\in\mathbb{Z}}\widehat{f_1}(p)_{n,m}\widehat{f_2}(p)_{m,k}.
\]
\end{lemma}
\begin{proof}
Using formula (\ref{star.p}), we get   
\begin{align*}
\widehat{f_1\star f_2}(p)_{n,k}
&=\left(\widehat{f_2}(p)\widehat{f_1}(p)\right)_{n,k}
=\langle\widehat{f_2}(p)\widehat{f_1}(p)\mathbf{e}_n,\mathbf{e}_k\rangle
=\langle\widehat{f_1}(p)\mathbf{e}_n,\widehat{f_2}(p)^*\mathbf{e}_k\rangle
\\&=\sum_{m\in\mathbb{Z}}\langle\widehat{f_1}(p)\mathbf{e}_n,\mathbf{e}_m\rangle\langle\mathbf{e}_m,\widehat{f_2}(p)^*\mathbf{e}_k\rangle
=\sum_{m\in\mathbb{Z}}\langle\widehat{f_1}(p)\mathbf{e}_n,\mathbf{e}_m\rangle\langle\widehat{f_2}(p)\mathbf{e}_m,\mathbf{e}_k\rangle
=\sum_{m\in\mathbb{Z}}\widehat{f_1}(p)_{n,m}\widehat{f_2}(p)_{m,k}.
\end{align*}
\end{proof}

\begin{proposition}
{\it Let $f_j\in L^1(SE(2))$ with $j\in\{1,2\}$ be admissible functions. Then $f_1\star f_2$ is admissible and 
\[
\beta_\infty(f_1\star f_2)\le\beta_\infty(f_1)\beta_\infty(f_2).
\]
}\end{proposition}
\begin{proof}
Let $f:=f_1\star f_2$,  $p>0$ and $n,k\in\mathbb{Z}$. Then, using Lemma \ref{convMat}, we have  
\[
|\widehat{f}(p)_{n,k}|\le\sum_{m\in\mathbb{Z}}|\widehat{f_1}(p)_{n,m}||\widehat{f_2}(p)_{m,k}|\le\sum_{m\in\mathbb{Z}}\|\widehat{f_1}_{n,m}\|_{\infty}\|\widehat{f_2}_{m,k}\|_\infty,
\]
implying that 
\[
\|\widehat{f}_{n,k}\|_\infty\le \sum_{m\in\mathbb{Z}}\|\widehat{f_1}_{n,m}\|_{\infty}\|\widehat{f_2}_{m,k}\|_\infty.
\]
Therefore,  we get 
\begin{align*}
\beta_k(f)&=\sum_{n\in\mathbb{Z}}\|\widehat{f}_{n,k}\|_\infty
\le\sum_{n\in\mathbb{Z}}\left(\sum_{m\in\mathbb{Z}}\|\widehat{f_1}_{n,m}\|_{\infty}\|\widehat{f_2}_{m,k}\|_\infty\right)
\\&=\sum_{m\in\mathbb{Z}}\beta_m(f_1)\|\widehat{f_2}_{m,k}\|_\infty
\le\beta_\infty(f_1)\sum_{m\in\mathbb{Z}}\|\widehat{f_2}_{m,k}\|_\infty
\le\beta_\infty(f_1)\beta_k(f_2)\le\beta_\infty(f_1)\beta_\infty(f_2),
\end{align*}
which guarantees that 
\[
\beta_\infty(f_1\star f_2)\le\beta_\infty(f_1)\beta_\infty(f_2).
\]
\end{proof}
\begin{theorem}
Let $f_j\in L^1\cap L^2(SE(2))$ with $j\in\{1,2\}$ be admissible functions such that $\widetilde{f_1}\in L^2(\mathbb{Z}^2\backslash SE(2))$.  Suppose $\mathbf{k}:=(k_1,k_2,k_3)\in\mathbb{Z}^3$.  
\begin{enumerate}
\item If $\rho(k_1,k_2)\not=0$ then 
\begin{align*}\label{kkk.conv}
\langle\widetilde{f_1}\oslash f_2,\psi_\mathbf{k}\rangle
=\sqrt{2\pi}e^{-\ii k_3\Phi(k_1,k_2)}
\sum_{n\in\mathbb{Z}}\sum_{m\in\mathbb{Z}}e^{-\ii n\Phi(k_1,k_2)}\widehat{f_1}(2\pi\rho(k_1,k_2))_{n,m}\widehat{f_2}(2\pi\rho(k_1,k_2))_{m,-k_3}.
\end{align*}
where the series converges absolutely.
\item If $\rho(k_1,k_2)=0$ then  
$$\langle\widetilde{f_1}\oslash f_2,\psi_\mathbf{k}\rangle=\sqrt{2\pi}\widehat{f_1}[k_3]\widehat{f_2}[k_3].$$
\end{enumerate}
\end{theorem}
\begin{proof}
Let $f:=f_1\star f_2$. Then, using Lemma \ref{conv12}, we get 
$f\in L^1\cap L^2(SE(2))$.
In addition,  applying Corollary \ref{conv2},  we conclude that $\widetilde{f}=\widetilde{f_1}\oslash f_2\in L^2(\mathbb{Z}^2\backslash SE(2))$.

(1) Using the assumption $f_j\in L^1\cap L^2(SE(2))$,  we get that $\widehat{f_j}(p)$ is a Hilbert-Schmidt operator,  for a.e. $p>0$. 
So,  we have  
\[
\|\widehat{f}(p)\|_1=\|\widehat{f_1\star f_2}(p)\|_1=\|\widehat{f_2}(p)\widehat{f_1}(p)\|_1\le\|\widehat{f_2}(p)\|_2\|\widehat{f_1}(p)\|_2<\infty,
\]
implying that for a.e.  $p>0$ the bounded linear operator $\widehat{f}(p)$ is trace-class.  Then 
\begin{align*}
\int_0^\infty\hspace{-0.1cm}\|\widehat{f}(p)\|_1p\dd p&\le \int_0^\infty\hspace{-0.1cm}\|\widehat{f_2}(p)\|_2\|\widehat{f_1}(p)\|_2p\dd p
\\&\le\left(\int_0^\infty\hspace{-0.1cm}\|\widehat{f_2}(p)\|_2^2p\dd p\right)^{1/2}\hspace{-0.1cm}\left(\int_0^\infty\hspace{-0.1cm}\|\widehat{f_1}(p)\|_2^2p\dd p\right)^{1/2}
\hspace{-0.3cm}=\|\widehat{f_2}\|_{\mathcal{H}^2}\|\widehat{f_1}\|_{\mathcal{H}^2},
\end{align*}
guarantees that $\widehat{f}\in\mathcal{H}^1(0,\infty)$.  
Therefore, applying formula (\ref{kkk}) for $f$,  we get  
\begin{align*}
\langle\widetilde{f_1}\oslash f_2,\psi_\mathbf{k}\rangle
&\hspace{-0.1cm}=\hspace{-0.1cm}\langle\widetilde{f_1\star f_2},\psi_\mathbf{k}\rangle
\hspace{-0.1cm}=\hspace{-0.1cm}\sqrt{2\pi}e^{-\ii k_3\Phi(k_1,k_2)}\hspace{-0.1cm}\sum_{n\in\mathbb{Z}}\hspace{-0.1cm}e^{-\ii n\Phi(k_1,k_2)}\widehat{f_1\star f_2}(2\pi\rho(k_1,k_2))_{n,-k_3}.
\end{align*}
By Lemma \ref{convMat},  for every $n\in\mathbb{Z}$, we have   
\[
\widehat{f_1\star f_2}(2\pi\rho(k_1,k_2))_{n,-k_3}=\sum_{m\in\mathbb{Z}}\widehat{f_1}(2\pi\rho(k_1,k_2))_{n,m}\widehat{f_2}(2\pi\rho(k_1,k_2))_{m,-k_3}
\]
So, we obtain 
\begin{align*}
&\langle\widetilde{f_1}\oslash f_2,\psi_\mathbf{k}\rangle
=\sqrt{2\pi}e^{-\ii k_3\Phi(k_1,k_2)}\hspace{-0.1cm}\sum_{n\in\mathbb{Z}}e^{-\ii n\Phi(k_1,k_2)}\hspace{-0.1cm}\left(\widehat{f_2}(2\pi\rho(k_1,k_2))\widehat{f_1}(2\pi\rho(k_1,k_2))\right)_{n,-k_3}
\\&=\sqrt{2\pi}e^{-\ii k_3\Phi(k_1,k_2)}\sum_{n\in\mathbb{Z}}\sum_{m\in\mathbb{Z}}e^{-\ii n\Phi(k_1,k_2)}\widehat{f_1}(2\pi\rho(k_1,k_2))_{n,m}\widehat{f_2}(2\pi\rho(k_1,k_2))_{m,-k_3}.
\end{align*}
(2) Using formulas (\ref{star.n}) and (\ref{00k}) for $f$,  we have 
\[
\langle\widetilde{f_1}\oslash f_2,\psi_\mathbf{k}\rangle=\sqrt{2\pi}\widehat{f}[k_3]=\sqrt{2\pi}\widehat{f_1}[k_3]\widehat{f_2}[k_3].
\]
\end{proof}

\begin{proposition}
{\it Let $f_j:SE(2)\to\mathbb{C}$ with $j\in\{1,2\}$ be rapidly decreasing functions.  Suppose $\mathbf{k}:=(k_1,k_2,k_3)\in\mathbb{Z}^3$.  
\begin{enumerate}
\item If $\rho(k_1,k_2)\not=0$ then 
\begin{align*}\label{kkk.conv.cts}
\langle\widetilde{f_1}\oslash f_2,\psi_\mathbf{k}\rangle
=\sqrt{2\pi}e^{-\ii k_3\Phi(k_1,k_2)}
\sum_{n\in\mathbb{Z}}\sum_{m\in\mathbb{Z}}e^{-\ii n\Phi(k_1,k_2)}\widehat{f_1}(2\pi\rho(k_1,k_2))_{n,m}\widehat{f_2}(2\pi\rho(k_1,k_2))_{m,-k_3}.
\end{align*}
where the series converges absolutely.
\item If $\rho(k_1,k_2)=0$ then  
$\langle\widetilde{f_1}\oslash f_2,\psi_\mathbf{k}\rangle=\widehat{f_1}[k_3]\widehat{f_2}[k_3]$.
\end{enumerate}
}\end{proposition}

\begin{corollary}
{\it Let $f_j:SE(2)\to\mathbb{C}$ with $j\in\{1,2\}$ be compactly supported smooth functions.  Suppose $\mathbf{k}:=(k_1,k_2,k_3)\in\mathbb{Z}^3$.  
\begin{enumerate}
\item If $\rho(k_1,k_2)\not=0$ then 
\begin{align*}
\langle\widetilde{f_1}\oslash f_2,\psi_\mathbf{k}\rangle
=\sqrt{2\pi}e^{-\ii k_3\Phi(k_1,k_2)}
\sum_{n\in\mathbb{Z}}\sum_{m\in\mathbb{Z}}e^{-\ii n\Phi(k_1,k_2)}\widehat{f_1}(2\pi\rho(k_1,k_2))_{n,m}\widehat{f_2}(2\pi\rho(k_1,k_2))_{m,-k_3}.
\end{align*}
where the series converges absolutely.
\item If $\rho(k_1,k_2)=0$ then  
$$\langle\widetilde{f_1}\oslash f_2,\psi_\mathbf{k}\rangle=\widehat{f_1}[k_3]\widehat{f_2}[k_3].$$
\end{enumerate}
}\end{corollary}

{\bf Conclusions.}
This paper has introduced a computational technique to investigate the mathematical theory of non-Abelian Fourier series on the compact right coset space $\mathbb{Z}^2\backslash SE(2)$ according to the concrete orthonormal basis consists of trigonometric polynomials.  As one of the main results,  a constructive closed form for coefficients of square-integrable functions of the form $\widetilde{f}$ on $\mathbb{Z}^2\backslash SE(2)$ presented,  in terms of non-Abelian Fourier matrix elements of $f$ on $SE(2)$.  There are several advantageous for this characterization of the non-Abelian Fourier coefficients which is benefited from the structure of the right group action of $SE(2)$ on the homogeneous space $\mathbb{Z}^2\backslash SE(2)$.  
To begin with,  the introduced closed form involves a unified discrete sampling of non-Abelian Fourier matrix elements from the interval $[0, \infty)$, which guarantees  a discrete spectrum of the computational method.   In addition,  appearance of non-Abelian Fourier matrix elements in the closed form formulation of coefficients implies compatibility of the characterization with the group convolution of $SE(2)$ on $\mathbb{Z}^2\backslash SE(2)$.  

{\bf Acknowledgments.} 
This research is supported by NUS Startup grants A-0009059-02-00 and A-0009059-03-00 and AME Programmatic Fund Project MARIO A-0008449-01-00. The authors gratefully acknowledge the supporting agencies. The findings and opinions expressed here are only those of the authors, and not of the funding agencies.

\bibliographystyle{amsplain}
\bibliography{TrigX}

\end{document}